\documentclass[10pt,reqno, final]{amsart}
\usepackage[colorlinks,linkcolor=blue,citecolor=blue,urlcolor=blue]{hyperref}
\usepackage{amsfonts,amssymb}
\usepackage{color}
\usepackage{graphicx}
\usepackage[notref,notcite]{showkeys}
\usepackage{mathtools}
\usepackage{extarrows}
\usepackage{algpseudocode}
\usepackage{algorithm}
\usepackage{caption}
\usepackage{subcaption}
\usepackage{tikz}
\usepackage{mathrsfs}
\usepackage{enumitem}

\usepackage{rotating}
\usepackage{rotfloat}
\usepackage{colortbl}
\usepackage{hhline}

\allowdisplaybreaks

\usepackage{todonotes}

\newtheorem{theorem}{Theorem}[section]
\newtheorem{lemma}[theorem]{Lemma}
\newtheorem{corollary}[theorem]{Corollary}
\newtheorem{proposition}[theorem]{Proposition}
\newtheorem{assumption}[theorem]{Assumption}

\theoremstyle{definition}
\newtheorem{definition}[theorem]{Definition}

\newtheorem*{example*}{Example}
\newtheorem{remark}[theorem]{Remark}

\numberwithin{equation}{section}

\usepackage{my_latex_commands}

\DeclareMathOperator*{\deltalim}{\Delta\text{\normalfont\textbf{--}}\lim}
\newcommand{\loc}{\textup{loc}}
\newcommand{\prox}{\operatorname{prox}}
\newcommand{\err}{\texttt{err}}

\begin{document}

\title[Finite Element Approximation of Data-Driven Problems]{Finite Element Approximation of 
Data-Driven Problems in Conductivity}
\thanks{This research was supported by the German Research Foundation (DFG) under grant 
number~ME~3281/10-1.}

\author{Annika M\"uller} \address{Technische Universit\"at Dortmund, Fakult\"at f\"ur
  Mathematik, Lehrstuhl LSX, Vogelpothsweg 87, 44227 Dortmund, Germany}
\email{annika.mueller@tu-dortmund.de}

\author{Christian Meyer} \address{Technische Universit\"at Dortmund, Fakult\"at f\"ur
  Mathematik, Lehrstuhl LSX, Vogelpothsweg 87, 44227 Dortmund, Germany}
\email{christian2.meyer@tu-dortmund.de}

\date{\today} 

\begin{abstract} 
    This paper is concerned with the finite element discretization of the data driven approach according to \cite{KirchdoerferOrtiz2016} 
    for the solution of PDEs with a material law arising from measurement data.
    To simplify the setting, we focus on a scalar diffusion problem instead of a problem in elasticity. 
    It is proven that the data convergence analysis from \cite{CMO2018} carries over to the finite element discretization 
    as long as $H(\div)$-conforming finite elements such as the Raviart-Thomas element are used. 
    As a corollary, minimizers of the discretized problems converge in data in the sense of \cite{CMO2018},
    as the mesh size tends to zero and the approximation of the local material data set gets more and more accurate.
    We moreover present several heuristics for the solution of the discretized data driven problems, which 
    is equivalent to a quadratic semi-assignment problem and therefore NP-hard. 
    We test these heuristics by means of two examples and it turns out that the ``classical'' alternating projection method 
    according to \cite{KirchdoerferOrtiz2016} is superior w.r.t.\ the ratio of accuracy and computational time. 
\end{abstract}

\maketitle


\section{Introduction}

In material science, empirically developed material models are commonly in use, i.e., 
material laws that describe the behavior of materials are derived from measured data. 
But, due to measuring errors and simplified models, this approach bears the risk of inaccuracies. 
For this reason, an alternative data-driven concept has been established in \cite{KirchdoerferOrtiz2016}. 
In a sense, this concept skips the modeling step and uses the measured data directly. 
The idea is to select that data point from the set of measurements that best fits 
axiomatic physical laws such as first principles. 

Let us 	explain this data-driven approach in terms of a stationary diffusion process of the form
\begin{equation}\label{SDP}
	- \div \kappa (\nabla u) = f \;\text{ in } \Omega, \quad u =0 \;\text{ on } \Gamma := \partial\Omega. 
\end{equation}
Here and in the following, $\Omega \subset \R^d$, $d\in \N$, is a bounded domain, 
$f \in H^{-1}(\Omega) \coloneqq H^1_0(\Omega)^*$, 
and $\div : L^2(\Omega; \R^d) \to H^{-1}(\Omega)$ denotes the distributional divergence.
Furthermore, $\kappa : \R^d \to \R^d$ is a given function which models the material law
and is calibrated by $m\in \N$ measurements for the tuple $(\bq,\nabla u)$ collected in the so-called 
\emph{local material data set}
\begin{equation}\label{eq:Dloc}
    \DD^{\loc} \coloneqq \{ (\br_1, \bw_1),\dots,(\br_m,\bw_m) \}  \subset \R^d \times \R^d.
\end{equation}
As indicated above, the idea is now to use these measurements directly. 
For this purpose, we rewrite \eqref{SDP} equivalently as 
\begin{equation*}
    \eqref{SDP} 
    \quad \Longleftrightarrow \quad 
    (\bq, \nabla u) \in \widetilde\DD\times \EE,
\end{equation*}
with the so-called \emph{equilibirum set}
\begin{equation}\label{eq:defequi}
    \EE \coloneqq \{ (\bq, \nabla u ) \in L^2(\Omega; \R^d) \times L^2(\Omega; \R^d) : 
    u \in H^1_0(\Omega),\;  - \div \sigma = f \}
\end{equation}
and the \emph{material law set}
\begin{equation}
    \widetilde\DD \coloneqq \{ (\br, \bw) \in L^2(\Omega; \R^d) \times L^2(\Omega; \R^d) : 
    \br(x) = \kappa(\bw(x)) \text{ a.e.\ in } \Omega\}.
\end{equation}
The data-driven approach now skips the modelling step and uses the measured data directly by 
replacing $\widetilde\DD$ with the \emph{material data set}
\begin{equation*}
    \DD \coloneqq \{ (\br, \bw) \in L^2(\Omega; \R^d) \times L^2(\Omega; \R^d) : 
    (\br, \bw) \in \DD^{\loc} \text{ a.e.\ in } \Omega\}
\end{equation*}
where $\DD^\loc$ is the collection of measurements from \eqref{eq:Dloc}.
Due to measurement errors and limited measuring capacities, the intersection 
$\DD \cap \EE$ is usually empty. One therefore resorts to a minimization problem of the form
\begin{equation} \tag{DDP} \label{DDP}
    \left.
    \begin{aligned}
        \min_{(y,z) \in Z \times Z} \quad &  \| y - z \|_Z^2 \\
        \text{s.t.} \quad \quad & y \in \DD, \, z \in \EE,
    \end{aligned}
    \quad \right\}
\end{equation}
i.e., one searches for two elements of the sets $\EE$ and $\DD$ that have smallest distance to each other. 
Herein, we abbreviated $Z := L^2(\Omega;\R^d) \times L^2(\Omega;\R^d)$.

The optimization problem \eqref{DDP} is frequently called \emph{data-driven problem}
and gives rise to several questions and issues:
First of all, while $\EE$ is easily seen to be weakly closed, the set $\DD$ is in general not.
Hence, \eqref{DDP} does not necessarily admit a solution. 
Moreover, a natural question arising in context of \eqref{DDP} is its behavior for measurements 
getting more and more accurate. Does the arising data-driven limit recover the ``true'' material law 
and, if so, in which sense? 
Moreover, a numerical solution of \eqref{DDP} requires a discretization of \eqref{DDP} and 
one may ask how a discretization influences this data-driven limit. 
Finally, the material data set $\DD$ involves a discrete point set such that \eqref{DDP} is
a mixed-integer optimization problem. Problems of this type are typically hard to handle
such that the development of efficient optimization algorithms for \eqref{DDP} (and its discretized 
counterpart) is all but trivial.

With this work, we address the two latter questions, i.e., 
we discuss discretization schemes and optimization algorithms for the solution of \eqref{DDP}. 
Let us put our work into perspective. 
In engineering science, the data-driven approach is meanwhile well accepted and has been applied 
to various problems, in particular in solid mechanics, we only refer to 
\cite{KirchdoerferOrtiz2016, Kirchdoerfer2017, NK18, Eggersmann2019, Carrara2020} for examples
from elasticity, inelasticity, dynamics, and fracture mechanics.
A rigorous mathematical analysis of this approach has only been initiated recently in \cite{CMO2018}, 
where the concept of \emph{data convergence} has been introduced. This notion of convergence 
is especially tailored to the structure of \eqref{DDP} and allows to characterize the data-driven limit. 
In this way, it answers the above question what happens, if the measurement errors tend to zero.
The precise characterization of data-driven limits strongly depends on the particular structure of $\DD$ and $\EE$. 
This notion of convergence and the characterization of the associated limits
have been investigated for several scenarios, we exemplarily refer to in \cite{CMO2018, CMO2020, RS20}.
To the best of our knowledge however, the discretization of $\EE$ and $\DD$ has not been 
incorporated into this convergence analysis so far and with this work, we aim to fill this gap.
This is an important issue, not only because 
\eqref{DDP} cannot be solved in infinite dimensional spaces, but also due to the general lack of existence of 
solutions to \eqref{DDP}. If one turns to a discretized counterpart of \eqref{DDP}, then the finite dimensional 
structure allows to establish the existence of optimal solutions under mild assumptions
so that, not until then, it makes sense to look for efficient algorithms for their computation.

As already indicated, the design of reliable and efficient solvers for \eqref{DDP} (and its discretization, 
respectively) is a delicate issue due to the discrete structure of the local material data set.
In \cite{KirchdoerferOrtiz2016}, a fixed-point type heuristic based on projections has been introduced, which 
is able to handle extensive measurement data, but may fail to converge or converges to 
spurious fixed-points that are not optimal, as shown in \cite{Kanno2019}. 
In the latter reference, a standard mixed-integer programming solver his employed to solve a data-driven problem
of (unrealistically) small size. Due to the vast amount of measurement data, it is in principle impossible 
to use exact mixed-integer programming solvers for the solution of \eqref{DDP}. 
Therefore, various heuristics have been developed and applied such as 
as kernel regression \cite{Kanno2018}, local regression \cite{Kanno_Local_Regression2018}, 
tensor voting \cite{Eggersmann2021}, and neural networks \cite{PBS21}.
In the second part of the paper, we present some new heuristics and compare them with existing methods. 
Some of our algorithms are based on projection heuristic from \cite{KirchdoerferOrtiz2016}, but we also 
tested a method, which employs an exact mixed-integer solver in combination with a local search 
algorithm.

We point out that we restrict ourselves to the conductivity example from \eqref{SDP} in order to keep 
the discussion as concise as possible. An extension of the finite element convergence analysis 
as well as the algorithmic approaches to problems in elasticity should be possible and is subject to future research.

The plan of the paper reads as follows: After introducing our standing assumptions and some well known 
results from saddle point theory in Section~\ref{sec:prelim}, we focus on the discretization of the 
equilibrium set by means of Raviart-Thomas type finite elements in Section~\ref{sec:RT}. 
Afterwards, in Section~\ref{sec:data}, we recall the notion of 
data convergence from \cite{CMO2018} and adapt it to our setting. 
Section~\ref{sec:conv} is then devoted to our main results, incorporating the finite element discretization 
of $\EE$ and $\DD$ into the data convergence analysis. In Section~\ref{sec:whyHdiv}, we 
discuss the need for $H(\div)$-conforming finite elements like the Raviart-Thomas element for 
the discretization of \eqref{DDP}. 
Section~\ref{sec:algo} is dedicated to the algorithms and their implementation and finally, 
in Section~\ref{sec:numerics}, we present some numerical results.


\section{Preliminaries and Standing Assumptions}\label{sec:prelim}

As usual we define 
\begin{equation*}
    H(\div) := \{ \bw \in L^2(\Omega;\R^d) \colon \div \bw \in L^2(\Omega)\},
\end{equation*}
where $\div: L^2(\Omega;\R^d) \to H^{-1}(\Omega)$ denotes the distributional divergence.
The following two lemmas concerning the space $H(\div)$ will be useful in the rest of the paper. 
For their proofs, we refer to \cite{temam}.

\begin{lemma}\label{lem:Temam}
    If the complement of $\overline{\Omega}$ satisfies the cone condition according to \cite{temam} , 
    then $C^\infty(\overline{\Omega};\R^d)$ is dense in $H(\div)$. 
\end{lemma}

	\begin{lemma}\label{lem:continuoussaddle}
		Let $F \in H(\div)^*$ and $f \in L^2(\Omega)$ be given. 
		Then there exists a unique solution $(\btau, \lambda) \in H(\div) \times L^2(\Omega)$ to the saddle point problem
		\begin{subequations}\label{eq:continuoussaddle}
			\begin{alignat}{3}
				& \int_\Omega \btau \cdot \bw \, dx + \int_\Omega \lambda \div \bw \, dx & \;
				= \; & \langle F , \bw \rangle \quad && \forall \, \bw \in H(\div), \label{eq:continuoussaddle_a} \\
				& - \int_\Omega v \div \btau \, dx & 
				= \; & \int_\Omega f v \, dx \quad && \forall v \in L^2(\Omega). \label{eq:continuoussaddle_b}
			\end{alignat}
		\end{subequations}
		If $\Omega$ satisfies the regularity assumptions from Lemma~\ref{lem:Temam} and $F \in L^2(\Omega; \R^d)$, 
		then $\lambda \in H^1_0(\Omega)$. If moreover $\Omega$ is $H^2$-regular and $F \in H^1(\Omega;\R^d)$, 
		then $\btau \in H^1(\Omega;\R^d)$ and
		\begin{align*}
			\| \btau \|_{H^1(\Omega; \R^d)} \leq c ( \| f \|_{L^2(\Omega)} + \| F \|_{H^1(\Omega; \R^d)} )
		\end{align*}
		with a constant $c>0$, which only depends on $\Omega$.
	\end{lemma}

\begin{assumption}\label{assu:domreg}
    Throughout this paper, we assume the following regularity assumptions on the domain
    and the inhomogeneity $f$:
    \begin{enumerate}[label=\textup{(\roman*)}]
        \item\label{it:domreg} The domain $\Omega$ is supposed to bounded and convex with a polygonal resp.\ 
        polyhedral boundary.
        \item The right-hand side in the divergence constraint satisfies $f\in L^2(\Omega)$.
    \end{enumerate}
\end{assumption}

\begin{remark}
    The regularity of the domain can be relaxed in the sense that we can drop the convexity, 
    see Remark~\ref{rem:domreg} below. 
    In contrast to this, we need the higher regularity of $f$ and cannot work with 
    inhomogeneities in $H^{-1}(\Omega)$, since our discretization requires $H(\div)$-conforming finite elements 
    as demonstrated in Section~\ref{sec:whyHdiv} below.    
\end{remark}

We underline that Assumption~\ref{assu:domreg} is tacitly supposed to hold throughout the rest of the paper 
without mentioning it every time.


\section{Discretization}\label{sec:RT}

The equilibrium constraint set $\EE$ from \eqref{eq:defequi} is discretized by 
\begin{equation}\label{eq:defEEh}
	\EE_h \coloneqq \{ (\bq_h, \nabla u_h) : \bq_h \in Q_h, u_h \in U_h, - \div_h \bq_h = f \},
\end{equation}
where $U_h$ and $Q_h$ are finite dimensional function spaces satisfying the following

\begin{assumption}\label{assu:discret}
    The discretization of $\EE$ in \eqref{eq:defEEh} is supposed to fulfill the following conditions:
    \begin{enumerate}[label=\textup{(\roman*)}]
        \item\label{it:conform}
        For all $h>0$, the discrete spaces $U_h$ and $Q_h$ are conforming, 
        i.e., they are finite dimensional linear subspaces of $H^1_0(\Omega)$ and $H(\div)$.
        Moreover, $\bigcup_{h > 0} U_h$ is dense in $H^1_0(\Omega)$ w.r.t.\ the $H^1(\Omega)$-norm.
        \item\label{it:divh}
        We set
        \begin{equation}\label{eq:defVh}
            V_h := \div(Q_h) \subset L^2(\Omega).
        \end{equation}
        Then the discrete divergence $\div_h : H(\div) \to V^*_h$ from \eqref{eq:defEEh} is defined by
        \begin{equation*}
            \langle \div_h \btau , v_h \rangle \coloneqq \int_\Omega v_h \div \btau \, dx, 
            \quad \btau \in H(\div), \, v_h \in V_h.
        \end{equation*}
        \item\label{it:interp}
        There exists an interpolation operator $\Pi_{h}: H^1(\Omega;\R^d) \to Q_{h}$ such that, 
        for all $\btau \in H^1(\Omega;\R^d)$ there holds
        \begin{equation}\label{eq:RTinterpol1}
            \div_h \btau  = \div_h \Pi_{h}\btau 
        \end{equation}
        and 
        \begin{equation}\label{eq:RTinterpol2}
            \Pi_h \btau \to \btau \quad \text{in } L^2(\Omega;\R^d)
            \text{ as } h \searrow 0.
        \end{equation}
    \end{enumerate}
\end{assumption}

Note that, in view of Assumtion~\ref{assu:discret}\ref{it:divh}, 
the constraint $- \div_h \bq_h = f$ in the definition of $\EE_h$ is short for
\begin{equation*}
	- \int_\Omega v_h \div \bq_h \, dx = \int_\Omega f \, v_h \, dx \quad \forall \, v_h \in V_h.
\end{equation*}
Assumption~\ref{assu:discret} implies the well known LBB-condition, as we shortly sketch in 
the following (by arguments analogous to the discussion after \cite[Lemma~5.4]{braess}). 
To this end, let $\bw\in H(\div)$ be arbitrary and solve \eqref{eq:continuoussaddle} 
with $(F, f) = (0, -\div \bw) \in H^1(\Omega;\R^d) \times L^2(\Omega)$. Now, since 
$\Omega$ is supposed to be convex and thus $H^2$-regular, 
we obtain a solution $(\btau, \lambda)$ with $\btau \in H^1(\Omega;\R^d)$. 
Thus we can apply the interpolation operator $\Pi_h$ from Assumption~\ref{assu:discret}\ref{it:interp}
and obtain 
\begin{equation}\label{eq:fortin}
    \div_h \bw = \div_h \btau = \div_h \Pi_h \btau.
\end{equation}
By construction, the mapping $H(\div) \ni \bw \mapsto \Pi_h \btau \in Q_h$ is 
linear and continuous and, in view of \eqref{eq:fortin}, it is a Fortin interpolation operator. 
Therefore, by \cite[4.8~Fortin's Criterion]{braess}, the tuple $(Q_h, V_h)$ satisfies the 
LBB-condition, i.e., we have shown the following:

\begin{corollary}\label{cor:fortin}
    Under Assumption~\ref{assu:discret}, $(Q_h, V_h)$ satiesfies the LBB-condition, i.e., 
    there exists a constant $\beta>0$, independent of $h>0$, such that
    \begin{equation}
        \inf_{v_h \in V_h} \sup_{\bw_h \in Q_h} 
        \frac{\int_\Omega v_h \div \bw_h \, dx}{\| \bw_h \|_{H(\div)} \|v_h \|_{L^2(\Omega)}} 
        \geq \beta.\label{eq:LBB}
    \end{equation}
\end{corollary}
 
Based on Assumption~\ref{assu:discret}\ref{it:conform}--\ref{it:divh}, 
the standard theory for mixed finite elements yields the following lemma.
For the corresponding proof, we refer to \cite[Section~4]{braess}.

\begin{lemma}\label{lem:saddleh}
    Let Assumption~\ref{assu:discret} be fulfilled. Then, the following is valid:
    \begin{enumerate}[label=\textup{(\roman*)}]
        \item\label{it:divdivh}
        For all $\bq_h, \btau_h \in Q_h$ satisfying $\div_h \bq_h = \div_h \btau_h$, 
        it holds that $\div \bq_h = \div \btau_h$.
        \item\label{it:saddlehexist}
        For every $F \in H(\div)^*$ and $f \in L^2(\Omega)$ 
        there exists a unique solution $(\btau_h, \lambda_h) \in Q_h \times V_h$ to the saddle point problem
        \begin{subequations} \label{eq:discretesaddle}
        \begin{alignat}{3}
            & \int_\Omega \btau_h \cdot \bw_h \, dx + \int_\Omega \lambda_h \div \bw_h \, dx 
            & \;	= \; &  \langle F , \bw_h \rangle \quad && \forall \, \bw_h \in Q_h \label{eq:discretesaddle_a} \\
            & - \int_\Omega v_h \div \btau_h \, dx & = \;  &  \int_\Omega f \, v_h \, dx 
            \quad && \forall v_h \in V_h. \label{eq:discretesaddle_b}
        \end{alignat}
        \end{subequations}
        \item\label{it:bestapprox1}
        This solution satisfies the following best approximation result
        \begin{equation}
            \| \btau - \btau_h \|_X 
            \leq 2 \inf \{ \| \btau - \bw_h \|_X : \bw_h \in Q_h, - \div_h \bw_h = f \}, \label{eq:bestapprox}
        \end{equation}
        where $\btau$ is the solution of \eqref{eq:continuoussaddle} and $X = H(\div)$ or $X = L^2(\Omega; \R^d)$.
        \item\label{it:bestapprox2}
        There exists a constant $C>0$ 
        depending only on the LBB-constant such that the solution of \eqref{eq:discretesaddle} satisfies.
        \begin{equation}\label{eq:aprioriboundh}
            \| \btau_h \|_{H(\div)} + \| \lambda_h \|_{L^2(\Omega)} \leq C (\| F \|_{H(\div)^*} + \| f \|_{L^2(\Omega)}).
        \end{equation}
        Moreover, the following best approximation result holds true
        \begin{equation}
            \| \btau - \btau_h \|_{H(\div)} \leq 2(1+C )\inf_{\bw_h \in Q_h} \| \btau - \bw_h \|_{H(\div)},
        \end{equation}
        where $\btau$ again denotes the solution of \eqref{eq:continuoussaddle}.
    \end{enumerate}     
\end{lemma}

\begin{remark}\label{rem:woLBB}
    Note that the assertions of Lemma~\ref{lem:saddleh}\ref{it:divdivh}--\ref{it:bestapprox1} 
    also hold without the LBB-condition, i.e., without the existence of the interpolation operator 
    from Assumption~\ref{assu:discret}\ref{it:interp}.
\end{remark}

\begin{lemma}\label{lem:distEE}
    Let Assumption~\ref{assu:discret} hold and let a sequence $\{(\bq_h, \nabla u_h)\}_{h>0}$ 
    with $(\bq_h, \nabla u_h) \in \EE_h$ be given. Then, there exists  a sequence 
    $\{(\hat \bq_h, \nabla \hat u_h)\}_{h>0} \subset \EE$ such that
	\begin{equation}\label{eq:strongconv}
        	\|(\hat \bq_h - \bq_h, \nabla (\hat u_h - \nabla u_h))\|_{L^2(\Omega;\R^d)^2} \to 0 
        	\quad \text{as } h \searrow 0.
	\end{equation}
\end{lemma}

\begin{proof}
	We consider the system \eqref{eq:continuoussaddle}  with $(F, f) = (0, f)$ 
	and denote the corresponding solution by $(\btau, \lambda) \in H(\div) \times L^2(\Omega)$.
    Moreover, we solve \eqref{eq:discretesaddle} with the same right hand side and denote this solution 
    by $(\btau_h, \lambda_h) \in Q_h \times V_h$. Furthermore, we set
	\begin{equation*}
		\hat\bq_h \coloneqq \bq_h + \btau- \btau_h \in H(\div).
	\end{equation*}
	Due to Lemma~\ref{lem:saddleh}\ref{it:divdivh}, there holds $\div (\bq_h - \btau_h) = 0$, 
	and, consequently, $\div \hat\bq_h = \div \btau = -f$ and therefore, $(\hat\bq_h, \nabla u_h) \in \EE$ 
	by the conformity of $U_h$ by Assumption~\ref{assu:discret}\ref{it:conform}. 
		
    Now, since $\Omega$ is convex and $f \in L^2(\Omega)$, Lemma~\ref{lem:continuoussaddle} implies 
    $\btau \in H^1(\Omega; \R^d)$ with $\| \btau \|_{H^1(\Omega;\R^d)} \leq c \| f \|_{L^2(\Omega)}$. 
    Thus, Lemma~\ref{lem:saddleh}\ref{it:bestapprox1} gives
	\begin{align*}
		\| \hat\bq_h - \bq_h \|_{L^2(\Omega; \R^d)} = \| \btau - \btau_h \|_{L^2(\Omega ; \R^d)} 
	    \leq 2 \| \btau - \Pi_h \btau \|_{L^2(\Omega; \R^d)} \to 0 
	    \quad \text{as } h\searrow 0.
	\end{align*}
	Therefore, if we set $\hat u_h := u_h$, we obtain $(\hat\bq_h, \nabla \hat u_h) \in \EE$ and 
	\eqref{eq:strongconv}.
\end{proof}

\begin{lemma} \label{lem:distEE_h}
    Let Assumption~\ref{assu:discret} be fulfilled and let $(\bq, \nabla u) \in \EE$ be given. 
    Then there is a sequence $\{ (\bq_h, \nabla u_h) \}_{h >0} \subset H(\div) \times L^2(\Omega; \R^d)$ 
    such that ${(\bq_h, \nabla u_h) \in \EE_h}$ and
    \begin{equation*}
        (\bq_h, \nabla u_h) \to (\bq, \nabla u) \quad \text{in } L^2(\Omega;\R^d) \times L^2(\Omega; \R^d)
        \quad \text{as } h \searrow 0.
    \end{equation*}
\end{lemma}
	
\begin{proof}
    Let $\varepsilon >0$ be arbitrary. By Lemma~\ref{lem:Temam}, there is a function 
    $\bq_\varepsilon \in C^\infty(\overline{\Omega}; \R^d)$ such that
    \begin{equation}
        \| \bq - \bq_\varepsilon \|_{H(\div)} \leq \min \{ 1, C^{-1} \} \frac{\varepsilon}{3}, 
        \label{eq:estimate_sigma_sigma_eps}
    \end{equation}
    where $C>0$ is the constant from Lemma~\ref{lem:saddleh}\ref{it:bestapprox2}. 
    Since $\bq_\varepsilon$ is smooth, we are allowed to apply $\Pi_h$, which yields
    \begin{equation*}
	    \| \bq_\varepsilon - \Pi_h \bq_\varepsilon \|_{L^2(\Omega; \R^d)} 
		\leq \frac{\varepsilon}{3}
	\end{equation*}
    provided that $h >0$ is chosen sufficiently small. Define now $f_\varepsilon \coloneqq - \div \bq_\varepsilon$ 
    and denote the solution of \eqref{eq:discretesaddle} with right hand side 
    $(0, f - f_\varepsilon)$ by $(\btau_h^\varepsilon, \lambda_h^\varepsilon)$. Then we set
    \begin{equation*}
        \bq_h^\varepsilon \coloneqq \Pi_h \bq_\varepsilon + \btau_h^\varepsilon \in Q_h.
    \end{equation*}
		Then, \eqref{eq:RTinterpol1} implies for every $v_h \in V_h$ that
		\begin{align*}
			\int_\Omega \div \bq_h^\varepsilon \, v_h \, dx 
			& = \int_\Omega \div(\Pi_h \bq_\varepsilon) \, v_h \, dx + \int_\Omega \div \btau_h^\varepsilon \, v_h \, dx \\
			& = \int_\Omega \div \bq_\varepsilon \, v_h \, dx - \int_\Omega (f - f_\varepsilon) \, v_h \, dx 
			= - \int_\Omega f \, v_h \, dx,
		\end{align*}
		i.\,e., $- \div_h \bq_h^\varepsilon = f$. 
		According to Lemma~\ref{lem:saddleh}\ref{it:bestapprox2}, 
		we deduce from \eqref{eq:estimate_sigma_sigma_eps} that
		\begin{equation*}
			\| \btau_h^\varepsilon \|_{L^2(\Omega; \R^d)} 
			\leq C \| f - f_\varepsilon \|_{L^2(\Omega)} 
			= C \| \div \bq - \div \bq_\varepsilon \|_{L^2(\Omega)} \leq \frac{\varepsilon}{3}.
		\end{equation*}
		Altogether, we obtain
		\begin{equation*}
			\| \bq - \bq_h^\varepsilon \|_{L^2(\Omega; \R^d)} 
			\leq \| \bq - \bq_\varepsilon \|_{L^2(\Omega; \R^d)} 
			+ \| \bq_\varepsilon - \Pi_h \bq_\varepsilon \|_{L^2(\Omega; \R^d)} 
			+ \| \btau_h^\varepsilon \|_{L^2(\Omega; \R^d)} \leq \varepsilon.
		\end{equation*}
		Finally, since $\bigcup_{h > 0} U_h$ is dense in $H^1_0(\Omega)$ by assumption, there exist $h>0$ 
		and $u_h \in U_h$ such that $\| \nabla u - \nabla u_h \|_{L^2(\Omega; \R^d)} \leq \varepsilon$. 
        As $\varepsilon >0$ was arbitrary, this proves the claim.
\end{proof}

\begin{proposition}\label{prop:RT}
    Let $\{\TT_h\}_{h>0}$ be a family of shape regular triangulations of $\Omega$
    according to \cite[Definition~5.1]{braess}.
    Then the Raviart-Thomas space of order $k\in \N \cup \{0\}$ given by
    \begin{equation*}
        \RR\TT_k(\TT_h) \coloneqq
        \{ \bw \in H(\div) \colon \bw |_T \in \RR T_k(T) \; \forall\, T \in \TT_h\}
    \end{equation*}        
    with $\RR\TT_k(T) \coloneqq \PP_k(T)^d + x \, \PP_k(T)$, where $\PP_k(T)$ denotes the space of polynomials 
    of order $k$ on $T$, is a feasible choice for $Q_h$ fulfilling Assumption~\ref{assu:discret}. 
    
    For $U_h$ one can choose the classical finite element space 
    \begin{equation*}
        U_h \coloneqq \{ u \in C(\overline{\Omega}) \cap H^1_0(\Omega) \colon u|_T \in \PP_k(T) \; \forall \, T \in \TT_h \} 
    \end{equation*}        
    in order to fulfill Assumption~\ref{assu:discret}.
\end{proposition}

\begin{proof}
    The conformity of $\RR\TT_k(\TT_h)$ and $U_h$ is already part of their definition. 
    The density of $\bigcup_{h>0} U_h$ in $H^1_0(\Omega)$ follows by smooth approximation and 
    standard interpolation error estimates. In case of Raviart-Thomas finite elements, the space $V_h = \div(Q_h)$ equals 
    $\PP_k(\TT_h) \coloneqq \{v \in L^2(\Omega) \colon v|_T \in  \PP_k(T) \; \forall \, T \in \TT_h \}$, 
    see e.g.\ \cite[Lemma~3.5]{duran}. 
    The existence of an interpolation operator $\Pi_h$ fulfilling Assumption~\ref{assu:discret}\ref{it:interp} 
    is established in \cite[Theorem~3.1, Lemma~3.5]{duran}.
\end{proof}

\begin{remark}
    There are several other elements satisfying Assumption~\ref{assu:discret}, for instance the 
    $\BB\DD\MM$-element or the Raviart-Thomas element on quadrilateral meshes.
    We refer to \cite{duran} and the references therein.
\end{remark}

\begin{remark}\label{rem:domreg}
    The regularity assumptions on $\Omega$ in Assumption~\ref{assu:domreg} can be relaxed. 
    In fact, it is sufficient to require that $\Omega$ is polygonally resp.\ polyhedrally bounded, i.e., 
    we can drop the convexity of $\Omega$. This is due to the fact that convexity is only needed for the 
    regularity of the solution of the saddle point problem \eqref{eq:continuoussaddle} 
    for the construction of Fortin's interpolation operator for Corollary~\ref{cor:fortin} and for the solution 
    $(\btau, \lambda)$ in the proof of \ref{lem:distEE}. In both cases however, one can resort to a larger 
    convex domain $B$ containing $\Omega$ and solve the continuous saddle point problem there
    such that the regularity result from Lemma~\ref{lem:continuoussaddle} applies.
    The function $\btau_h$ in the proof of Lemma~\ref{lem:distEE} is then defined on $B$ 
    and for this reason, one needs to assume that the meshes can be extended in a shape regular way from 
    $\Omega$ to $B$ so that the results of Lemma~\ref{lem:saddleh} also hold on $B$ instead of $\Omega$.
    In order to avoid these technical issues, we restrict ourselves to the case of a convex domain $\Omega$.
\end{remark}


\section{Data Topology}\label{sec:data}

Let us recall the concept of data convergence, which was first introduced in \cite{CMO2018}.
It represents an intermediate convergence between weak and strong convergence and is especially tailored 
to the structure of the data driven problem \eqref{DDP}.

	\begin{definition}[Data convergence]
		Let $Z$ be a reflexive, separable Banach space. A sequence $\{(y_k,z_k)\}_{k \in \N}$ in $Z \times Z$ is said to converge to $(y,z) \in Z \times Z$ in the data topology, denoted $(y,z) = \deltalim_{k \to \infty} (y_k,z_k)$, if
		\begin{align*}
			y_k \rightharpoonup y, \quad z_k \rightharpoonup z \quad \text{and} \quad  y_k-z_k \to y-z 
			\quad \text{in } Z.
		\end{align*}
	\end{definition}
	
	The concept of data convergence can be transferred to sets.

	\begin{definition}[Data convergence of sets]
		Let $Z$ be a reflexive, separable Banach space and $\DD, \DD_k, \EE, \EE_k \subset Z$, $k \in \N$. We write
		$\DD \times \EE = \deltalim_{k \to \infty} (\DD_k \times \EE_k)$, if
		\begin{enumerate}[label=(DC\arabic*)]
			\item\label{it:DC1}
			for each $(y,z) \in \DD \times \EE$ there is a sequence $\{(y_k,z_k)\}_{k \in \N}$ with $(y_k,z_k) \in \DD_k \times \EE_k$ for each $k \in \N$ such that $(y,z) = \deltalim_{k \to \infty} (y_k,z_k)$,
			\item\label{it:DC2}
			for each sequence $\{(y_{j},z_{j})\}_{j \in \N}$ with $(y_{j},z_{j}) \in \DD_{k_j} \times \EE_{k_j}$ for each ${j} \in \N$, $\{k_j\}_{j \in \N}$ strictly monotonically increasing, and $(y,z) = \deltalim_{j \to \infty} (y_{j},z_{j})$ it holds that $(y,z) \in \DD \times \EE$.
		\end{enumerate}
	\end{definition}

	Note that the above definition of data convergence of sets corresponds to Kuratowski convergence of sets with respect to data convergence.
The notion of data convergence of sets is especially well suited to the approximation of data-driven problems 
of the form \eqref{DDP}, as the following proposition shows. 
Its proof is along the lines of \cite[Theorem~3.2]{CMO2018}, where the equilibrium set $\EE$ is fixed.
Here we  additionally consider the approximation of $\EE$ is with a sequence of sets $\EE_k$. 
Though the proof is a straightforward
adaptation of the one in \cite{CMO2018}, we present it for convenience of the reader.

\begin{proposition} \label{prop:solution_convergence}
    Let $Z$ be a reflexive and separable Banach space and suppose that subsets $\DD, \EE\subset Z$
    and sequences of subsets $\{\DD_k\}_{k\in \N}$, $\{\EE_k\}_{k\in\N}$,     
    $\DD_k, \EE_k \subset Z$ for all $k\in \N$, are given such that 
    \begin{equation}\label{eq:dataconv}
        \DD \times \EE = \deltalim_{j \to \infty} (\DD_{k} \times \EE_{k}).
    \end{equation}
    Assume moreover that there are constants $c >0$ and $b \geq 0$, independent of $k\in \N$, such that, 
    for all $k\in \N$,
    \begin{equation}\label{eq:equitrans}
        \| y-z \|_Z \geq c \big( \|y\|_Z + \|z \|_Z \big) - b   \quad \forall\, (y, z) \in \DD_k \times \EE_k.
    \end{equation}        
    Furthermore, define $F_k : Z\times Z \to [0,\infty]$ by
    \begin{equation*}
        F_k(y,z) \coloneqq I_{\DD_k}(y) + I_{\EE_k}(z) + \| y - z \|_Z^2,
    \end{equation*}
    where $I_{\DD_k} : Z \to \{0, \infty\}$ is the indicator functional of $\DD_k$, i.e., 
    \begin{equation*}
        I_{\DD_k}(y) \coloneqq 
        \begin{cases}
            0 , & y \in \DD_k, \\
            \infty ,& y \notin \DD_k
        \end{cases}
    \end{equation*}
    and $ I_{\EE_k}$ is defined analogously.
    Then, the following is valid:
    \begin{enumerate}[label=\textup{(\alph*)}]
        \item\label{it:liminf}
        If $F_k(y_k,z_k) \to 0$, there exists $z \in \DD\cap \EE$ such that, 
        up to subsequences, $(z,z) = \deltalim_{k \to \infty}(y_k,z_k)$;
        \item\label{it:limsup}
        If $z \in \DD \cap \EE$, there exists a sequence $\{(y_k,z_k)\}_{k \in \N}$ in $Z \times Z$ 
        such that $(z,z) = \deltalim_{k \to \infty}(y_k,z_k)$ and $F_k(y_k,z_k) \to 0$.
    \end{enumerate}
\end{proposition}
	
\begin{proof}
    ad \ref{it:liminf}: Let $F_k(y_k,z_k) \to 0$. Then, it follows that $y_k \in \DD_k, z_k \in \EE_k$ for $k$ sufficiently large 
    and $\|y_k-z_k \| \to 0$ as $k \to \infty$. By \eqref{eq:equitrans}, 
    $\{y_k\}_{k \in \N}$ and $\{z_k\}_{k \in \N}$ are bounded. 
    Therefore, there are subsequences $\{y_{k_j}\}_{j \in \N}$ and $\{z_{k_j}\}_{j \in \N}$ 
    and $y \in Z$ and $z \in Z$ such that $y_{k_j} \weakly y$ and $z_{k_j} \weakly z$. 
    By weak lower-semicontinuity of the norm, we have that
    \begin{equation*}
        0 \leq \| y-z \|_Z \leq \liminf_{j \to \infty} \| y_{k_j} - z_{k_j} \|_Z = 0.
    \end{equation*}
    Hence $y=z$ and $(z,z) = \deltalim_{j \to \infty}(y_{k_j},z_{k_j})$ and therefore, \eqref{eq:dataconv} yields
    $z\in \DD \cap \EE$ as claimed.
			
    ad \ref{it:limsup}: Let $z \in \DD \cap \EE$ be given. Then, thanks to \eqref{eq:dataconv}, 
    there exists a sequence $\{(y_k,z_k)\}_{k \in \N}$ with
    \begin{equation*}
        (y_k,z_k) \in \DD_k \times \EE_k
        \quad \text{and} \quad 
        (z,z) = \deltalim_{k \to \infty} (y_k,z_k).
    \end{equation*}
    This in particular implies $y_k-z_k \to z-z = 0$ and hence, by continuity of the norm,
    \begin{equation*}
        \lim_{k \to \infty} F_k(y_k,z_k) 
        = \lim_{k \to \infty} \left( I_{\DD_k}(y_k) + I_{\EE_k} (z_k) + \| y_k - z_k \|_Z^2 \right) =0,
    \end{equation*}
    as required.
\end{proof}

Proposition~\ref{prop:solution_convergence} shows that, if a sequence of sets $\{(\DD_k, \EE_k)\}_{k\in \N}$ 
satisfies \eqref{eq:equitrans} and more importantly \eqref{eq:dataconv}, then 
the data-driven problem with limit sets $\DD$ and $\EE$ admits a solution, which can be approximated
(w.r.t.\ data convergence) with solutions of the respective data-driven problems subject to the sets 
$\DD_k$ and $\EE_k$.
The crucial question is of course now, which (sequences of sets) satisfy \eqref{eq:dataconv}.
This will be answered for our conductivity example in the following section.


\section{Convergence Results}\label{sec:conv}

As in \cite[Theorem~3.3]{CMO2018}, we aim at giving sufficient conditions 
under which the assumptions of Proposition~\ref{prop:solution_convergence} are fulfilled. 
Recall again the setting in our conductivity example, where 
\begin{equation}\label{eq:defZ}
    Z = L^2(\Omega;\R^d) \times L^2(\Omega;\R^d)
\end{equation}
and
\begin{equation}\label{eq:defEE}
    \EE = \{ (\bq, \nabla u ) \in L^2(\Omega; \R^d) \times L^2(\Omega; \R^d) : 
    u \in H^1_0(\Omega), - \div \bq = f \}.
\end{equation}
For the approximation of $\EE$, we choose the discretized equilibrium constraint sets $\EE_{h_k}$ 
from \eqref{eq:defEEh}. The following theorem shows that such a discretization
can be included in the convergence analysis of \cite[Theorem~3.3]{CMO2018}.

\begin{theorem}\label{thm:data_lim}
    Let $Z$ and $\EE$ be given as in \eqref{eq:defZ} and \eqref{eq:defEE}, respectively, and 
    assume that a global material data set $\DD\subset Z$ and approximations thereof, 
    denoted by $\DD_k\subset Z$, $k\in \N$, are given. 
    Suppose moreover the following to hold:
    \begin{enumerate}[label=\textup{(\roman*)}]
        \item\label{it:dataclos}
        \emph{(Data closure)} 
        $\overline{\DD} \times \EE = \overline{\DD \times \EE}^\Delta$, 
        i.\,e., $\overline{\DD} \times \EE$ is the closure of $\DD \times \EE$ w.r.t.\ data convergence;
        \item\label{it:fineapprox}
        \emph{(Fine approximation)} 
        For each $\xi \in \DD$, 
        there is a sequence $\{ \xi_k \}_{k \in \N}$ with $\xi_k \in \DD_k$ for all $k \in \N$ such that
        $\xi_k \to \xi$ as $k \to \infty$;
        \item\label{it:uniapprox}
        \emph{(Uniform approximation)}
        There is a sequence $\{t_k\}_{k \in \N} \subset \R_{>0}$ with $t_k \searrow 0$ such that
        \begin{equation*}
            d(\xi,\DD) \coloneqq \inf_{y\in \DD} \|y-\xi\|_Z \leq t_k \quad \forall\, \xi \in \DD_k;        
        \end{equation*}
        \item\label{it:trans}
        \emph{(Transversality)}
        There are constants $c >0$ and $b \geq 0$ such that, for all $y \in \DD$ and $z \in \EE$,
        \begin{equation*}
            \| y - z \|_Z \geq c \big( \| y \|_Z + \| z \|_Z \big) - b ;
        \end{equation*}
	    \item\label{it:confdisc}
	    \emph{(Conforming discretization)} 
	    There is a monotonically decreasing sequence of mesh sizes $\{h_k\}_{k \in \N} \subset \R_{>0}$ 
	    with $h_k \searrow 0$ as $k \to \infty$ such that the discrete spaces $Q_k \coloneqq Q_{h_k}$ 
	    and $U_k \coloneqq U_{h_k}$ from the discrete equilibrium set $\EE_k \coloneqq \EE_{h_k}$ in 
	    \eqref{eq:defEEh} satisfy Assumption~\ref{assu:discret}. 
    \end{enumerate}
		Then the assumptions of Proposition~\ref{prop:solution_convergence} are fulfilled, i.e., 
		\begin{enumerate}[label=\textup{(\alph*)}]
			\item\label{it:dataconv}
			\emph{(Data convergence)} 
			$\overline{\DD} \times \EE = \deltalim_{k \to \infty} (\DD_{k} \times \EE_{k})$;
			\item\label{it:equitrans}
			\emph{(Equi-transversality)} 
			There are constants $c >0$ and $b \geq 0$ such that, for all $k\in \N$ and all  
			$(y, z) \in \DD_k \times \EE_{k}$, there holds
			\begin{align*}
				\| y - z \|_Z \geq c \big( \| y \|_Z + \| z \|_Z \big) - b.
			\end{align*}
		\end{enumerate}
	\end{theorem}
	
\begin{proof}
ad \ref{it:dataconv}, condition~\ref{it:DC1}: 
Let $(y,z) \in \overline{\DD} \times \EE$ be fixed but arbitrary. 
Our goal is to find a sequence $\{(y^*_k,z^*_k)\}_{k \in \N}$ with $(y^*_k,z^*_k) \in \DD_k \times \EE_{k}$ 
such that $(y,z) = \deltalim_{k \to \infty} (y^*_k,z^*_k)$. 
By \ref{it:dataclos}, there is a sequence $\{(\hat{y}_n, \hat{z}_n)\}_{n \in \N} \subset \DD \times \EE$ 
such that $(y,z) = \deltalim_{n \to \infty}(\hat{y}_n, \hat{z}_n)$. 
Due to \ref{it:fineapprox} and Lemma~\ref{lem:distEE_h}, for each $n \in \N$, there are sequences $\{y_{n,k}\}_{k \in \N}$ 
with $y_{n,k} \in \DD_k$ and $\{z_{n,k}\}_{k \in \N}$ with $z_{n,k} \in \EE_{k}$ 
and a finite number $m_n \in \N$ with $m_n \geq m_{n-1}+1$ such that
\begin{equation*}
    \| y_{n,k} - \hat{y}_n \|_Z < \frac{1}{n} \quad \text{and} \quad \| z_{n,k} - \hat{z}_n \|_Z < \frac{1}{n}
    \quad \forall\, k \geq m_n.
\end{equation*}
This of course gives rise to a diagonal sequence $\{ y_{n, m_n}, z_{n, m_n} \}$ with the desired properties, but, 
for each $(y,z) \in \overline{\DD} \times \EE$, one obtains a different sequence $\{m_n\}_{n\in \N}$ with 
different approximations $\DD_{m_n}$ and discretizations $\EE_{m_n}$.
To overcome this issue, let us define
\begin{equation*}
    \{(\hat{y}^*_k,\hat{z}^*_k)\}_{k \in \N} 
    \coloneqq \{ \underbrace{(\hat{y}_1,\hat{z}_1), \dots, (\hat{y}_1,\hat{z}_1)}_{\text{$(m_1-1)$-times}}, 
    \underbrace{(\hat{y}_1,\hat{z}_1), \dots, (\hat{y}_1,\hat{z}_1)}_{\text{$(m_2-m_1)$-times}} , 
    \underbrace{(\hat{y}_2,\hat{z}_2), \dots, (\hat{y}_2,\hat{z}_2)}_{\text{$(m_3-m_2)$-times}}, \dots \}
\end{equation*}
as well as
\begin{align*}
    \{ (y^*_k,z^*_k) \}_{k \in \N} \coloneqq \{ &(y_{1,1},z_{1,1}),\dots,(y_{1,m_1-1},z_{1,m_1-1}), \\
    & (y_{1,m_1}, z_{1,m_1}), \dots, (y_{1,m_2-1},z_{1,m_2-1}), \\
    & (y_{2,m_2},z_{2,m_2}), \dots, (y_{2,m_3-1},z_{2,m_3-1}), \dots \}.
\end{align*}
Then, by construction, $(y^*_k,z^*_k) \in \DD_k \times \EE_{k}$ for all $k \in \N$. 
Moreover, we have
\begin{equation}\label{eq:deltalimhat}
    (y,z) = \deltalim_{k \to \infty} (\hat{y}^*_k, \hat{z}^*_k)
\end{equation}
and, since, for each $n\in \N$ and all $k \geq m_{n}$, it holds
\begin{equation*}
    \| \hat{y}^*_k - y^*_k \|_Z \leq \frac{1}{n} \quad \text{and} \quad \| \hat{z}^*_k - z^*_k \|_Z \leq \frac{1}{n} ,
\end{equation*}
we obtain 
\begin{equation}\label{eq:convkstar}
    \| \hat{y}^*_k - y^*_k \|_Z \to 0 \quad \text{and} \quad 
    \| \hat{z}^*_k - z^*_k \|_Z \to 0 \quad \text{as } k \to \infty.
\end{equation}
By the definition of data convergence, \eqref{eq:deltalimhat} and \eqref{eq:convkstar} yield 
$y^*_k \weakly y$, $z^*_k \weakly z$, and
\begin{equation*}
    \| y^*_k - z^*_k - (y - z) \|_Z
    \leq \| y^*_k - \hat{y}^*_k \|_Z + \| \hat{y}^*_k - \hat{z}^*_k - (y-z) \|_Z + \| \hat{z}^*_k - z^*_k \|_Z \to 0,
\end{equation*}
which is nothing else than
\begin{equation*}
    (y,z) = \deltalim_{k \to \infty} (y^*_k, z^*_k)
\end{equation*}
with $ (y^*_k, z^*_k)\in \DD_k \times \EE_{k}$ for all $k \in \N$. Since $(y, z) \in \overline{\DD} \times \EE$ was
arbitrary, this implies \ref{it:DC1}.

ad (a), condition~\ref{it:DC2}: 
Suppose that $(y,z) = \deltalim_{j \to \infty} (y_j,z_j)$ in $Z \times Z$ 
with $(y_j,z_j) \in \DD_{k_j} \times \EE_{k_j}$ for all $j \in \N$ and a strictly monotonically increasing sequence 
$\{k_j\}_{j\in \N}$. We need to prove $(y,z) \in \overline{\DD} \times \EE$.
By \ref{it:uniapprox} and Lemma~\ref{lem:distEE}, there exist $\hat{y}_j \in \DD$ and $\hat{z}_j \in \EE$ 
such that 
\begin{equation*}
    \| \hat{y}_j - y_j \|_Z \leq t_{k_j} \quad \text{and} \quad \|\hat{z}_j - z_j \|_Z \to 0
    \text{ as } j \to \infty. 
\end{equation*}
Consequently, $\hat{y}_j \rightharpoonup y$, $\hat{z}_j \rightharpoonup z$ and $\hat{y}_j - \hat{z}_j \to y - z$
so that $(y,z) = \deltalim_{j \to \infty} (\hat{y}_j , \hat{z}_{j})$. 
Thus \ref{it:dataclos} implies $(y,z) \in \overline{\DD} \times \EE$.
					
ad (b): Let $k\in \N$ and $(y,z) \in \DD_k \times \EE_{h_k}$ be arbitrary. 
By the uniform approximation property \ref{it:uniapprox}, there is $\hat{y} \in \DD$ with $\| y - \hat{y} \|_Z < t_k$ 
and by Lemma~\ref{lem:distEE} there exists $\hat{z} \in \EE$ with 
$\| z - \hat{z} \|_Z \leq c\, h_k \| f \|_{L^2(\Omega)} =: r_k$. Therefore, \ref{it:trans} implies
\begin{equation*}
    \| y - z \|_Z \geq c \big( \| y \|_Z + \| z \|_Z \big) - b - (c+1) r_k - (c+1) t_k.
\end{equation*}
Since the sequences $t_k$ and $r_k$ are bounded, equi-transversality holds with
$b' \coloneqq b + (1+c)  (\max_{k \in \N} t_k + \max_{k \in \N} r_k)$.
\end{proof}
	
\begin{remark}\label{rem:dataclos}
    Since $\EE$ as defined in \eqref{eq:defEE} is closed and convex and thus weakly closed and 
    data convergence implies weak convergence, the set $\EE$ itself arises in the data closure 
    in \ref{it:dataclos}. The situation changes, if one turns to the material data set $\DD$. 
    Of course, if the constitutive law coupling $\bq$ and $\nabla u$ is linear such as 
    in case of Fourier's law for instance, $\DD$ is weakly closed, too, such that 
    $\DD \times \EE = \overline{\DD \times \EE}^\Delta$. By contrast, if the constitutive law is nonlinear, then 
    $\overline{\DD}$ will in general differ from the $Z$-closure of $\DD$, but also from the closure of its convex hull.
    The latter is due to the fact that data convergence provides more information than just weak convergence.
    The computation of data closures is a field of active research, we only refer to \cite{RS20} and the references therein.
\end{remark}	

So far we have focused on the discretization of the set $\EE$. A possible discretization of the set $\DD$ is given by piecewise constant functions. To fulfill condition \ref{it:fineapprox} of Theorem~\ref{thm:data_lim}, 
we need to bound the distance between the values of those piecewise constant functions 
and the values of the functions in $\DD$ by a monotonically decreasing sequence that converges to zero, 
which is done in the following 

\begin{proposition}\label{prop:DDconv}
    Let a monotonically decreasing sequence $\{h_k\}_{k \in \N} \subset \R_{>0}$ with ${h_k \to 0}$     
    and a corresponding sequence of shape regular triangulations $\TT_{h_k}$ of $\Omega$ be given. Let
    \begin{equation}\label{eq:DDex}
	    \DD \coloneqq \{ y \in Z : y(x) \in \DD^{\loc} \text{ a.e.\ in } \Omega \}
    \end{equation}
    with $\DD^{\loc} \subset \R^{d} \times \R^d$ and
    \begin{equation}\label{eq:DDexk}
        \DD_k \coloneqq \{ y \in Z : y(x)  \in \DD^{\loc}_k \text{ a.e.\ in } \Omega, \; 
        y \vert_T \in \PP_0(T) \;\forall \, T \in \TT_{h_k} \}
    \end{equation}
    with $\DD^{\loc}_k \subset \R^d \times \R^d$ be given. 
    Moreover, assume that there is a sequence
    $\{\rho_k\}_{k \in \N}$ with $\rho_k \searrow 0$ such that the Hausdorff distance between 
    $\DD^{\loc}$ and $\DD^{\loc}_k$ satisfies 
    \begin{equation}\label{eq:DDlocapprox}
        d_{\textup H}(\DD^{\loc}, \DD^{\loc}_k)
        = \max\Big\{ \sup_{\xi \in \DD^{\loc}} d(\xi, \DD^{\loc}_k), 
        \sup_{\eta \in \DD^{\loc}_k}  d(\eta, \DD^{\loc}) \Big\}
		\leq \rho_k.	
    \end{equation}
    Then $\DD$ and $\DD_k$ as defined in \eqref{eq:DDex} and \eqref{eq:DDexk}, respectively, 
    satisfy the fine and uniform approximation assumption \ref{it:fineapprox} and \ref{it:uniapprox} 
    in Theorem~\ref{thm:data_lim}.
\end{proposition}
	
	\begin{proof}
		Define 
		$V_{h_k} \coloneqq \{ v: \Omega \to \R^d \times \R^d : v \vert_T \equiv \mathrm{const.} \; \forall \, T \in \TT_{h_k} \}$. 
		Then, by standard interpolation error analysis, 
		$\bigcup_{k \in \N} V_{h_k}$ is dense in $Z$. 
		Therefore, for $y \in \DD$ and $\varepsilon >0$ fixed, but arbitrary, 
		there exist $k^\varepsilon_1 \in \N$ such that for all $k \geq k^\varepsilon_1$
		there is a $v_k \in V_{h_k}$ with 
		\begin{equation}\label{eq:y_vk}
			\| y - v_k \|_Z \leq \varepsilon.
		\end{equation}
		Define $\bar{y}_k \in V_{h_k}$, $k \in \N$, by
		\begin{align*}
		    & \bar y_k(x) \in \DD^{\loc} \quad \text{a.e.\ in }\Omega,\\
			& \|\bar{y}_k(x) - v_k(x)\|_{\R^d \times \R^d} 
			\leq \inf_{\xi \in \DD^{\loc}} \| \xi - v_k(x) \|_{\R^d\times \R^d} + \varepsilon
			\quad \text{a.e.\ in } x \in \Omega.
		\end{align*}
		Note that $\bar y_k$ is well defined, since $v_k$ is constant on each $T\in \TT_{h_k}$.
		Then, $\bar{y}_k \in \DD$ and
		\begin{equation}\label{eq:baryk_vk}
			\| v_k - \bar{y}_k \|_Z \leq \| v_k - y \|_Z + \sqrt{|\Omega|} \, \varepsilon
		\end{equation}
		for all $k \geq k^\varepsilon_1$. 
		Moreover, there is $k^\varepsilon_2 \in \N$ such that $\rho_k \leq \varepsilon$ for all $k \geq k^\varepsilon_2$.
		Hence, according to \ref{eq:DDlocapprox}, for each $k \geq k^\varepsilon_2$, there is 
		an $y_k \in \DD_k$ such that
		\begin{equation}\label{eq:baryk_yk}
			\| \bar{y}_k - y_k \|_Z \leq \sqrt{|\Omega|} \, \varepsilon.
		\end{equation}
       Altogether, \eqref{eq:y_vk}--\eqref{eq:baryk_yk} yield
       $\| y - y_k \|_Z \leq (1 + 2\sqrt{|\Omega|}) \varepsilon$ for all 
       $k \geq \max \{k^\varepsilon_1, k^\varepsilon_2\}$, which along with $y_k \in \DD_k$ implies \ref{it:fineapprox}.
		
	To verify \ref{it:uniapprox}, let now $k\in \N$ and $y_k \in \DD_k$ be fixed, but arbitrary. 
	Then, by definition of $\DD_k$, there exist 
	$\xi^{(k)}_T \in \DD^{\loc}_k$, $T \in \TT_{h_k}$, such that 
	$y_k = \sum_{T\in \TT_{h_k}} \xi^{(k)}_T \, \chi_T$ a.e.\ in $\Omega$. 
	In view of \eqref{eq:DDlocapprox}, for every $T$, we find $\xi_T \in \DD^{\loc}$    	
	such that $\|\xi_T - \xi_T^{(k)}\|_{\R^d \times \R^d} \leq \rho_k$. Therefore, if we define 
    $y\in \DD$ by $y \coloneqq \sum_{T\in \TT_{h_k}} \xi_T \, \chi_T$,
    then 
    \begin{equation*}
        \| y_k - y \|_Z^2 = \sum_{T\in \TT_{h_k}} \int_T \|\xi_T - \xi_T^{(k)}\|_{\R^d \times \R^d}^2 \, dx
        \leq |\Omega| \, \rho_k^2,
    \end{equation*}
    which is \ref{it:uniapprox} with $t_k = \sqrt{|\Omega|} \, \rho_k$.
\end{proof}

Let us denote the number of elements in $\TT_{h_k}$ by $N_k \coloneqq |\TT_{h_k}|$. 
Then $\DD_k$ as defined in \eqref{eq:DDexk} is isomorphic to the finite dimensional set 
\begin{equation*}
    \D_k \coloneqq \{ A \in \R^{N_k \times d \times d} : A_i \in \DD^{\loc}_k \; \forall \, i = 1, ..., N_k\},
\end{equation*}
which is clearly compact provided that $\DD^{\loc}_k$ is so. This observation immediately implies 
the following

\begin{proposition}\label{prop:Pkexists}
    Suppose that $\DD$, given as in \eqref{eq:DDex}, is discretized as in Proposition~\ref{prop:DDconv} with 
    approximate local material data sets $\DD_k^{\loc}$ that are compact for every $k\in \N$. 
    Assume moreover, that the equilibrium constraint set is discretized as in \eqref{eq:defEEh} with spaces 
    $Q_{h_k}$ and $U_{h_k}$ satisfying Assumption~\ref{assu:discret}. 
    Then, for each $k\in \N$, the discretized data-driven problem given by
    \begin{equation}\tag{P$_k$} \label{eq:Pk}
        \left. 
        \begin{aligned}
            \min \quad &  \tfrac{1}{2} \| y-z \|_Z^2 \\
            \text{\textup{s.t.}} \quad & y \in \DD_k, \, z \in \EE_{k}
        \end{aligned}
        \quad \right\}
    \end{equation}   
    admits a globally optimal solution.  
\end{proposition}

\begin{proof}
    Throughout the proof, let us suppress the index $k$ in $h_k$ to simplify the notation.
    First we rewrite \eqref{eq:Pk} as 
    \begin{equation*}
    \eqref{eq:Pk}
    \quad \Longleftrightarrow \quad 
    \left\{\quad 
    \begin{aligned}
        \min_{(\br, \bw)\in \DD_k} \quad
        \min_{(\bq, u)}   \quad &
        \tfrac{1}{2}\,\|(\nabla u_h, \bq_h) - (\bw, \br)\|_{L^2(\Omega;\R^d)^2}^2\\
         \text{s.t.} \quad & \bq_h \in Q_{h}, \; \,u_h \in U_{h},\;\,
         - \div_h \bq_h = f.
    \end{aligned}
    \right.
    \end{equation*}
    By standard arguments, the direct method of calculus of variations yields the existence and uniqueness 
    of a solution to the inner minimization problem. Due to strict convexity, it is uniquely characterized 
    by its necessary and sufficient conditions, which read as follows: 
    Thanks to the surjectivity of $\div_h: Q_h \to V_h^*$, a tuple 
    $(\bq_h, u_h)$ is a solution of the inner minimization problem, 
    iff there exists a Lagrange multiplier $\lambda_h \in V_h$ such that 
    \begin{subequations}\label{eq:KKTh}
    \begin{alignat}{3}
        \int_\Omega \nabla u_h \cdot \nabla \varphi_h\, dx & = 
        \int_\Omega \bw \cdot \nabla \varphi_h\, dx & \quad & \forall\, \varphi_h \in U_{h}
        \label{eq:laplacediscr}\\
        \int_\Omega (\bq_h \cdot \bw_h + \lambda_h \,  \div \bw_h )\, dx
        &= \int_\Omega \br \cdot \bw_h\, dx  && \forall \, \bw_h \in Q_{h} \label{eq:saddle_a}\\
        - \int_\Omega v_h \, \div \bq_h \, dx &=  \int_\Omega  f \, v_h \, dx  && \forall v_h \in V_{h}. 
        \label{eq:saddle_b} 
    \end{alignat}        
    \end{subequations}
    By Lemma~\ref{lem:saddleh}\ref{it:saddlehexist}, the saddle point system \eqref{eq:saddle_a}--\eqref{eq:saddle_b}
    admits a unique solution $(\bq_h, \lambda_h) \in Q_{h} \times V_{h}$ for every right hand side 
    $(\br, f) \in L^2(\Omega;\R^d) \times L^2(\Omega)$ and the associated solution operator is linear and continuous 
    by \eqref{eq:aprioriboundh}. Moreover, since $U_h$ is a closed subspace of $H^1_0(\Omega)$, 
    the same holds for the discretized Laplace equation in \eqref{eq:laplacediscr}, 
    i.e., for every $\bw \in L^2(\Omega;\R^d)$, 
    there is a unique solution $u_h \in U_{h}$ and the solution mapping is linear and continuous.
    Thus, there is an affine (due to $f$) and continuous solution operator of \eqref{eq:KKTh} denoted by 
    \begin{equation}\label{eq:defGh}
        G_h : L^2(\Omega;\R^d)^2 \ni (\br, \bw) \mapsto (\bq_h, \nabla u_h) \in Q_h \times \nabla U_h. 
    \end{equation}
    With the help of $G_h$, we can rewrite \eqref{eq:Pk} equivalently as 
    \begin{equation*}
        \eqref{eq:Pk} 
        \quad \Longleftrightarrow \quad 
        \min_{(\br, \bw)\in \DD_k} \; \tfrac{1}{2}\,\|(G_h - \id)(\br, \bw)\|_{L^2(\Omega;\R^d)^2}^2
    \end{equation*}
    Therefore, since $\DD_k$ is compact as explained above and $G_h$ and thus the whole objective 
    is continuous, the existence of a globally optimal solution follows from the Weierstrass theorem.
\end{proof}

\begin{remark}
    Note that the result of Proposition~\ref{prop:Pkexists} also holds for a problem of the form
    \begin{equation}\tag{$\tilde{\textup P}_k$} 
        \left. 
        \begin{aligned}
            \min \quad &  \tfrac{1}{2} \| y-z \|_Z^2 \\
            \text{\textup{s.t.}} \quad & y \in \DD_k, \, z \in \EE
        \end{aligned}
        \quad \right\}
    \end{equation}
    with the continuous equilibrium set $\EE$ instead of $\EE_k$. The arguments are completely the same as 
    in the proof of Proposition~\ref{prop:Pkexists}, since the solution operators associated with the continuous 
    counterparts to the saddle point problem and the Laplace equation are also linear and continuous.
    Thus, to ensure the mere existence of optimal solutions, only the discretization of $\DD$ is necessary,
    whereas the numerical computation of optimal solutions of course requires a discretization of $\EE$, too.
\end{remark}

As a direct consequence of the previous results, namely Theorem~\ref{thm:data_lim} and 
Propositions~\ref{prop:DDconv} and \ref{prop:Pkexists}, we obtain the following result, which, 
though it is just a corollary as a consequence of the above findings, can be seen as our main result:
	
\begin{corollary}\label{cor:data_lim}
    Let $Z$, $\EE$, and $\DD$ be defined as in \eqref{eq:defZ}, \eqref{eq:defEE}, and \eqref{eq:DDex}.
    Furthermore, let $\{h_k\}_{k \in \N} \subset \R_{>0}$ be a monotonically decreasing sequence of mesh sizes 
    with $h_k \searrow 0$ as $k \to \infty$ and suppose the following assumptions:
    \begin{enumerate}[label=\textup{(\roman*)}]
        \item
        \emph{(Transversality)}
        There are constants $c >0$ and $b \geq 0$ such that, for all $y \in \DD$ and $z \in \EE$,
        \begin{equation*}
            \| y - z \|_Z \geq c \big( \| y \|_Z + \| z \|_Z \big) - b.
        \end{equation*}
        \item \emph{(Discrete material data set)}
        The set $\DD$ is approximated as in Proposition~\ref{prop:DDconv}, i.e., 
        \begin{equation*}
            \DD_k \coloneqq \{ y \in Z : y(x) \in \DD^{\loc}_k \text{ a.e.\ in } \Omega, \; 
            y \vert_T \in \PP_0(T) \;\forall \, T \in \TT_{h_k} \}
        \end{equation*}
        and the Hausdorff distance fulfills 
        $d_{\textup H}(\DD^{\loc}, \DD^{\loc}_k) \leq \rho_k$ with a 
        sequence $\{\rho_k\}_{k \in \N}$ with $\rho_k \searrow 0$.
        Moreover, $\DD^{\loc}_k \subset \R^d \times \R^d$ is compact for every $k\in \N$.
	    \item
	    \emph{(Conforming discretization)} 
	    The finite dimensional spaces $Q_{k}$ and $U_{k}$ defining the discrete equilibrium set $\EE_k$
	    satisfy Assumption~\ref{assu:discret}. 
    \end{enumerate}
    Then, for every $k\in \N$, there exists at least one solution of  
    \begin{equation}\tag{P$_k$}\label{eq:Pk2}
        \left. 
        \begin{aligned}
            \min \quad &  \tfrac{1}{2} \| y-z \|_Z^2 \\
            \text{\textup{s.t.}} \quad & y \in \DD_k, \, z \in \EE_{k}
        \end{aligned}
        \quad \right\}
    \end{equation}
    and every sequence $\{(y_k,z_k)\} \subset Z \times Z$ of such solutions satisfies the following:
    \begin{enumerate}[label=\textup{(\alph*)}]
        \item If $\overline{\DD \times \EE}^\Delta \neq \emptyset$, then $\| y_k - z_k \|_Z \to 0$.
        \item If $\|y_k-z_k\|_Z \to 0$, then there exists $z \in \overline{\DD \times \EE}^\Delta$ such that, 
        up to subsequences, $(z,z) = \deltalim_{k \to \infty}(y_k,z_k)$.
    \end{enumerate}
\end{corollary}

We again underline that the data closure $\overline{\DD \times \EE}^\Delta$ is in general not equal to 
$\DD \times \EE$, but $\overline{\DD} \times \EE$ with an enlarged set $\overline{\DD} \supset \DD$, 
see Remark~\ref{rem:dataclos}. 
Corollary~\ref{cor:data_lim} shows that, under the mentioned assumptions, 
it is in principle possible to approximate elements of the data closure.
The computation of solutions to \eqref{eq:Pk2} may however be a very delicate issue, depending 
on the precise structure of $\DD^{\loc}_k$.
Before we address this issue in more details in the Section~\ref{sec:algo} below, let us shortly discuss the question
why the use of $H(\div)$-conforming finite elements like the Raviart-Thomas 
element seems to be indispensable for the discretization of the data-driven problem \eqref{DDP}.


\section{Why $H(\div)$-conforming finite elements?}\label{sec:whyHdiv}

As the construction of $H(\div)$-conforming finite elements is rather complicated compared 
to e.g.\ classical Lagrangian finite elements, the question arises if their use is really necessary 
for the discretization of data-driven problems. 
To answer this question, let us assume that we do not use $H(\div)$-conforming elements, i.e., 
we discretize the equilibrium set $\EE$ by a finite dimensional space $\widehat{Q}_h \subset L^2(\Omega;\R^d)$ with
$\widehat{Q}_h \not\subset H(\div)$ and define the discrete divergence condition by
\begin{equation}\label{eq:divhatQ}
    \bq_h \in \widehat{Q}_h, \quad
    \int_\Omega \bq_h \cdot \nabla v_h \,dx = \int_\Omega f\,v_h\, dx 
    \quad \forall\, v_h \in \widehat{V}_h, 
\end{equation}
where $\widehat{V}_h$ is a finite dimensional subspace of $H^1_0(\Omega)$.
Given a triangulation $\TT_h$ of the domain $\Omega$, 
a classical example for these finite dimensional spaces reads
\begin{align}
    \widehat{Q}_h &= \{ \bw \in L^\infty(\Omega;\R^d) \colon \bw|_{T} \in \PP_0(T)\; \forall\, T \in \TT_h\}, 
    \label{eq:P0}\\
    \widehat{V}_h &= \{ v \in C(\overline\Omega) \cap H^1_0(\Omega)\colon v|_{T} \in \PP_1(T)\; \forall\, T \in \TT_h\}.
    \label{eq:classicFE}
\end{align}

Revisiting the proof of Theorem~\ref{thm:data_lim} shows that 
a central aspect of convergence analysis is that elements from $\EE$ can be approximated 
by elements from $\EE_h$ w.r.t.\ the data topology and vice versa.
Let us return to part \ref{it:dataconv} in the proof of Theorem~\ref{thm:data_lim}. 
When verifying condition~\ref{it:DC2} from the definition of data convergence, one
considers a sequence $\{(y_h, z_h)\}_{h>0}$, $(y_h, z_h) \in \DD_{h} \times \EE_{h}$, 
converging in data to $(y,z)$. To show that $(y,z)$ is an element of the data closure, 
we need to prove the existence of a sequence $(\hat y_h, \hat z_h) \in \DD \times \EE$ with 
$\hat y_h \weakly y$, $\hat z_h \weakly z$, and $\hat y_h - \hat z_h  \to y - z$, which, in view of the 
data convergence of $\{(y_h, z_h)\}$ is equivalent to
\begin{equation}\label{eq:dataconvdiff}
  \hat y_h - y_h \weakly 0, \quad \hat z_h - z_h \weakly 0, \quad \hat y_h - y_h + z_h - \hat z_h \to 0.   
\end{equation}
If we assume that the data approximation satisfies the uniform approximation property \ref{it:uniapprox} 
from Theorem~\ref{thm:data_lim}, then we already know that a sequence $\{\hat y_h\}_{h>0} \subset \DD$ exists with 
$\hat y_h - y_h \to 0$ in $Z = L^2(\Omega;\R^d)^2$. Therefore, if we take this sequence for $\hat y_h$, then 
the sequence $\{\hat z_h\}_{h>0}\subset \EE$ must necessarily fulfill 
\begin{equation}\label{eq:zstrong}
    \hat z_h - z_h \to 0 \quad \text{in } Z
\end{equation}
in order to guarantee \eqref{eq:dataconvdiff}. This however is not always possible, if 
nonconforming finite element spaces are used, as we will see in the following when
considering the flux component of 
\begin{equation*}
    z_h = (\bq_h, \nabla u_h) \in \EE_{h} \subset (\widehat{Q}_{h}, \nabla U_{h}) .
\end{equation*}
The best possible choice for the construction of the desired elements from $\EE$ is of course to choose the 
solution $\hat \bq_h \in H(\div)$ of 
\begin{equation}
    \left.
    \begin{aligned}
        \min_{\bq \in L^2(\Omega;\R^d)} \quad & \tfrac{1}{2} \|\bq - \bq_h\|_{L^2(\Omega;\R^d)}^2\\
        \text{s.t.} \quad & - \div \bq = f,
    \end{aligned}
    \quad \right\}
\end{equation}
which is uniquely characterized by the existence of $w\in H^1_0(\Omega)$ such that 
\begin{equation*}
    \hat \bq_h - \bq_h + \nabla w = 0, \quad - \div \hat \bq_h = f.
\end{equation*}
Let us define $\Phi \in H^1_0(\Omega)$ by 
\begin{equation*}
    - \laplace \Phi = f \quad \text{in } H^{-1}(\Omega),
\end{equation*}
as well as the $L^2$-projection of $\bq_h$ on $\nabla H^1_0(\Omega)$, denoted by $\nabla \Phi^h$
with $\Phi^h \in H^1_0(\Omega)$. Then, we obtain
\begin{equation}
\begin{aligned}
    \| \hat \bq_h  - \bq_h\|_{L^2(\Omega\;\R^d)} 
    &= \|\nabla w\|_{L^2(\Omega\;\R^d)} \\
    &= \sup_{\substack{v\in H^1_0(\Omega)\\ \|\nabla v\|_{L^2(\Omega;\R^d)} \leq 1}}
    \int_\Omega \nabla (\Phi - \Phi^h) \cdot \nabla v\, dx\\
    &= \|\nabla \Phi - \nabla \Phi^h\|_{L^2(\Omega;\R^d)}.
\end{aligned}
\end{equation}
Since $\widehat V_h$ is a closed subspace of $H^1_0(\Omega)$, we may 
decompose $\Phi^h =  \Phi^h_0 + \Phi^h_\perp \in  \widehat V_h \oplus \widehat V_h^\perp$, 
where the orthogonal complement is taken w.r.t.\ the $H^1_0$-scalar product. 
Due to $- \div_h \bq_h = f$, we find for $\Phi^h_0$ 
\begin{equation}\label{eq:Phi0eq}
    \int_\Omega \nabla \Phi^h_0 \cdot \nabla v_h \, dx 
    = \int_\Omega \bq_h \cdot \nabla v_h \, dx = \dual{f}{v_h} 
    \quad \forall\, v_h \in \widehat{V}_h
\end{equation}
and consequently, by the best approximation property of the finite element solution,
\begin{equation}\label{eq:Phi0conv}
    \|\Phi_0^h - \Phi\|_{H^1_0(\Omega)} \to 0 \quad \text{as } h\searrow 0,
\end{equation}
follows, provided that $\bigcup_{h>0}\widehat{V}_h$ is dense in $H^1_0(\Omega)$.
Therefore, if the sequence $\{\bq_h\}_{h>0}$ is such that 
\begin{equation}\label{eq:Phiperpnorm}
    \liminf_{h\searrow 0}\|\nabla \Phi^h_\perp\|_{L^2(\Omega;\R^d)} \geq c > 0,
\end{equation}
then \eqref{eq:Phi0conv} implies
\begin{equation}\label{eq:disthatqq}
\begin{aligned}
     & \liminf_{h\searrow 0} \| \hat \bq_h  - \bq_h\|_{L^2(\Omega\;\R^d)} \\
     & \qquad = \liminf_{h\searrow 0}  \|\nabla \Phi - \nabla \Phi^h\|_{L^2(\Omega;\R^d)} \\
     & \qquad \geq \liminf_{h\searrow 0}  \Big(\|\nabla \Phi^h_\perp\|_{L^2(\Omega;\R^d)} 
     - \|\nabla \Phi - \nabla \Phi^h_0\|_{L^2(\Omega;\R^d)}\Big) \geq c
\end{aligned}
\end{equation}
and hence, \eqref{eq:zstrong} cannot hold in this case.

Let us give a simple one-dimensional example indicating that \eqref{eq:Phiperpnorm} may well happen
for a sequence converging in data. For this purpose, set $\Omega := (0,2\pi)$ and consider the
equidistant triangulation of $\Omega$ given by 
\begin{equation*}
    \TT_k := \{(0, h_k), (h_k, 2h_k), ..., ((2k-1)h_k, 2\pi)\}
\end{equation*}
with $h_k := (2 k)^{-1} \,2\pi$, $k\in \N$.
We employ the standard finite element space from \eqref{eq:classicFE} and denote the nodal 
basis of $\widehat{V}_k := \widehat{V}_{h_k}$ by $\varphi^k_i$, 
i.e., $\varphi^k_i(j \,h_k) = \delta_{ij}$, $1\leq i,j \leq 2 k - 1$. 
Note that the meshes and thus the spaces $\widehat{V}_k$ are nested. Furthermore, we 
abbreviate the solution of \eqref{eq:Phi0eq} by $\Phi_0^k := \Phi_0^{h_k} \in \widehat{V}_k$ 
and set 
\begin{equation}\label{eq:defek1}
    \Psi_k(x) := \frac{1}{\sqrt{2\pi}} \sum_{i=1}^{2k} \varphi^{2k}_{2i-1}(x)
    \in H^1_0(\Omega), \quad k\in \N.
\end{equation}
Then one easily verifies that 
\begin{equation}\label{eq:Psik}
    \|\tfrac{d}{dx} \Psi_k\|_{L^2(0,2\pi)} = 1, \quad 
    \tfrac{d}{dx}\Psi_k \weakly 0\quad \text{as } k\to \infty
\end{equation}
and 
\begin{equation*}
    \int_0^{2\pi} \tfrac{d}{dx} \Psi_k(x) \, \tfrac{d}{dx} \varphi^k_i(x)\, dx = 0
    \quad \forall \, i = 1, ..., 2k-1, \; k\in \N. 
\end{equation*}
Since $\widehat{V}_k = \spanv(\varphi^k_1, ..., \varphi^k_{2k-1})$, 
the latter implies $\Psi_k \in \widehat V_k^\perp$. 
Let us choose the space from \eqref{eq:P0} for $\bq_h$, but with half mesh size, i.e., 
\begin{equation*}
    \widehat Q_k := \spanv( \chi_{(0, h_{2k})}, \chi_{(h_{2k}, 2 h_{2k})}, ..., \chi_{((4k-1)h_{2k}, 2\pi)} )
\end{equation*}
and set 
\begin{equation*}
    \bq_{k} := \tfrac{d}{dx}\Phi_0^k + \tfrac{d}{dx} \Psi_k \in \widehat{Q}_k
\end{equation*}
such that $\Phi_\perp^k = \Psi_k$.
Then, owing to \eqref{eq:Phi0eq} and \eqref{eq:Psik}, we have 
\begin{equation*}
    -\div_{h} \bq_k = f, \quad \bq_k \weakly \tfrac{d}{dx} \Phi \quad  \text{in } L^2(0,2\pi),
\end{equation*}
Together with $\bq_k \in \widehat{Q}_k$, 
this indicates that $\bq_k$ can be part of a sequence converging in data, depending on the structure 
of $\DD_k$. Indeed, if, for instance, there is an $a\in \R$ such that 
$(a, 1/\sqrt{2\pi}), (-a, -1/\sqrt{2\pi}) \in \DD_k^\loc$ for all $k\in \N$, then 
one could choose $U_{k} := \widehat{V}_{2k}$, $u_k := a \sum_{i=1}^{2k} \varphi^{2k}_{2i-1}$,
and set $y_k = (\br_k, \bw_k) = (\frac{d}{dx} \Psi_k, \frac{d}{dx} u_k) \in \DD_k$. 
Then $y_k \weakly 0$ and 
\begin{equation*}
    z_k - y_k = (\bq_k - \br_k, \nabla u_k - \bw_k) = (\tfrac{d}{dx}\Phi_0^k, 0)
    \to (\tfrac{d}{dx}\Phi, 0) \quad \text{in } Z,
\end{equation*} 
i.e., $((0,0),(\tfrac{d}{dx}\Phi, 0)) = \deltalim_{k\to \infty} (y_k, z_k)$.
However, due to \eqref{eq:disthatqq} and $\Phi_\perp^k = \Psi_k$, there holds
\begin{equation}\label{eq:bqk2}
    \liminf_{k\to \infty} \|\hat\bq_k - \bq_k\|_{L^2(0,2\pi)} 
    \geq \lim_{k\to\infty} \|\tfrac{d}{dx} \Psi_k\|_{L^2(0,2\pi)} = 1
\end{equation}
such that \eqref{eq:zstrong} does not hold in this example.

Altogether we have seen that it is in general not possible to construct a sequence 
$\{\hat z_h\}_{h>0} \subset \EE$ such that \eqref{eq:zstrong} holds, if a finite element space 
$\widehat{Q}_h$ is used that is not $H(\div)$-conforming. In contrast to this, 
Lemma~\ref{lem:distEE} shows that this is well possible in case of $H(\div)$-conforming finite elements.
Nonetheless, this is no proof that no sequences $\{(\hat y_h, \hat z_h)\}_{h>0} \subset \DD \times \EE$ 
exist such that \eqref{eq:dataconvdiff} holds also in case of nonconforming finite element spaces. 
Maybe, it is possible to construct a sequence $\hat z_h$ by adjusting the sequence $\hat y_h$, but, so far, 
we have no idea how to do this and this issue gives rise to future research.


\section{Algorithms}\label{sec:algo}

The last section is devoted to optimization algorithms for the numerical solution of the discretized data-driven problem \eqref{eq:Pk2}.
In the literature, a projection based fixed-point method is frequently used to solve \eqref{eq:Pk2}.
We will advance this method by introducing a step size
and compare the proximal gradient method arising in this way with two variants of the Douglas-Rachford algorithm applied to \eqref{eq:Pk2}. 
Besides these first-order methods, we also employ the equivalence of \eqref{eq:Pk2} to a quadratic semi-assignment problem and 
develop a heuristic based on this reformulation. 
The algorithms are tested by means of two examples, a linear and a non-linear material model.

\subsection{Fixed-point methods based on projections}\label{sec:prox}

We first consider a class of algorithms that is based on the two $L^2$-projections $\pi_{\EE_k} : Z \to \EE_k$ and $\pi_{\DD_k} : Z \to \DD_k$ 
with $\EE_{k} \coloneqq \EE_{h_k}$ as in \eqref{eq:defEEh} and $\DD_{k}$ as in Proposition \ref{prop:DDconv}, respectively. 
As the proof of Proposition~\ref{prop:Pkexists} shows, $\pi_{\EE_k}$ is simply given by the affine operator $G_h$ defined in \eqref{eq:defGh}, i.e., 
given $y = (\br,\bw) \in Z$, we have $\pi_{\EE_k}(\br,\bw) = (\bq_{h_k}, \nabla u_{h_k})$, 
where $(\bq_{h_k}, u_{h_k}) \in Q_{h_k} \times U_{h_k}$ is the unique solution to the saddle point system \eqref{eq:KKTh}.
To compute the projection $\pi_{\DD_k}$, we first notice that we can project an arbitrary function $z \in Z$ onto $\DD_k$ 
by first projecting $z$ onto the space of piecewise constant functions and then projecting the obtained function onto $\DD_k$. 
To see this, define $\hat z := \frac{1}{|T|} \int_T z \, dx \,\chi_T$ for a given function $z\in Z$. Then the above assertion follows from
\begin{equation}\label{eq:piDk}
\begin{aligned}
	\underset{y \in \DD_k}{\arg \min} \, \| y - z \|_Z^2 & = \underset{y \in \DD_k}{\arg \min} \, \|y\|_Z^2 - 2 (y,z) + \| z \|_Z^2 \\
	& = \underset{y \in \DD_k}{\arg \min} \, \| y \|_Z^2 - 2 (y,\hat{z}) + \| \hat{z} \|_Z^2 
	= \underset{y \in \DD_k}{\arg \min} \, \| y - \hat{z} \|_Z^2.
\end{aligned}
\end{equation}
Moreover, the structure of $\DD_k$ yields that the projection can be computed elementwise, 
for example by using a $k$-nearest neighbor search \cite{Friedman1977AlgBestMatches}.
It is to be noted that $\pi_{\DD_k}$ is a set-valued map, since the projection on $\DD_k$ is clearly non-unique in 
general due to the non-convexity of $\DD_k$. In the numerical realization of the algorithms introduced below, we pick an
arbitrary $y\in \DD_k$ attaining the minimum in \eqref{eq:piDk}.

In \cite{KirchdoerferOrtiz2016}, a simple projection algorithm was introduced to solve the data driven problem.
One iteration step consists of two projections and reads
\begin{align*}
	y_{n+1} = \pi_{\DD_k} \big(\pi_{\EE_k}  (y_n)\big) \tag{PG} \label{eq:P}
\end{align*}
for $n \in \N$ and an arbitrary initial point $y_0 \in Z$. The algorithm terminates if $y_{n+1} = y_n$ for some $n \in \N$.
Note that this algorithm is a special case of the proximal gradient method or more specifically of the projected gradient method \cite{Beck2017,CV2020,NatemeyerWachsmuth2021} with step size constant equal to one. 
To illustrate this, let us shortly recall the concept of the proximal gradient method. 
Given to mappings $F, G: Z \to \R\cup\{\infty\}$, where $F$ is smooth, 
while $G$ may be not, the proximal gradient iteration for the minimization of $F+G$ reads
\begin{equation}\label{eq:proxit}
    y_{n+1} = \prox_{\gamma G}\big( y_n - \gamma\,\nabla F(y_n) \big),
\end{equation}
where $\prox_{\gamma G}(y) \in \argmin_{\eta \in Z} \tfrac{1}{2} \|\eta  - y\|_{Z}^2 + \gamma\, G(\eta)$ denotes the proximal map and
$\gamma > 0$ is a step size.
In order to apply the proximal gradient method to the data driven problem, let us rewrite \eqref{eq:Pk2} as
\begin{equation}\label{eq:Pkreform}
    \eqref{eq:Pk2}
    \quad \Longleftrightarrow \quad
    \min_{y\in Z} \; \frac{1}{2} \| \pi_{\EE_k}(y) - y \|_{Z}^2 + I_{\DD_k}(y),
\end{equation}
where $I_{\DD_k}$ denotes the indicator functional of $\DD_k$.
If we now set $F(y) := \frac{1}{2} \| \pi_{\EE_k}(y) - y \|_{Z}^2$ and $G := I_{\DD_k}$, then \eqref{eq:proxit} becomes
\begin{equation}\tag{PS} \label{eq:PS}
	y_{n+1} = \pi_{\DD_k}\big(y_n - \gamma (y_n - \pi_{\EE_k} (y_n) )\big) .
\end{equation}
If $F$ and $G$ were proper, convex, and lower semi-continuous, then the classical convergence results for proximal gradient methods
apply provided that $0 < \gamma_{\min} \leq \gamma < 2L^{-1}$, where $L$ is the Lipschitz constant of $\nabla F$, 
cf.\ e.g.\ \cite[Theorems 9.6 and 10.2]{CV2020}.
In our case, $\nabla F(y) = y - \pi_{\EE_k}(y)$ has the Lipschitz constant $L=1$ and, accordingly, we have chosen $\gamma \in [1,2)$
in our computations.
We emphasize that $G = I_{\DD_k}$ is clearly not convex due to the non-convexity of $\DD_k$ and therefore, 
convergence of the iteration in \eqref{eq:PS} to global minimizers cannot be guaranteed. 
A convergence analysis of the proximal gradient method for non-convex problems in function space similar to \eqref{eq:Pk2} 
is presented in \cite{Wachsmuth2019}. However, due to the lack of convexity, the convergence results are rather limited 
and one only obtains that weak accumulation points of the sequence of iterates satisfy a rather weak stationarity concept termed L-stationarity, 
provided that $\nabla F$ is completely continuous.
One can therefore not expect the iterates produced by the scheme in \eqref{eq:PS} to converge to 
a minimizer (neither global nor local) of \eqref{eq:Pk2}. 
Nonetheless, the introduction of the step size $\gamma$ may improve the quality of the numerical results. 
In our numerical tests, it has turned out that it is favorable to start with $\gamma = 1.4$ 
and to reduce $\gamma$ whenever the algorithm circles between two iterates, that is $y_{n+2} = y_n$ for some $n \in \N$, 
see Table~\ref{tab:gamma} below.
The choice $\gamma < 1$ accelerates the convergence of the algorithm to a fixed point 
with substantially larger objective value and should therefore be avoided.

We compare the performance of the proximal gradient method (with step size $\gamma = 1$ and varying step size) 
with the Douglas-Rachford algorithm for the computation of elements in the intersection of two sets $A, B \subset Z$,
cf., e.g., \cite{ACT20}. This algorithm is also initialized with an arbitrary point $y_0 \in Z$ and the iterates are computed by the formula
\begin{align*}
	y_{n+1} = T_{A,B}(y_n)\quad \text{with}\quad 
	T_{A,B} \coloneqq \frac{\id + R_B \circ R_A}{2}, 
\end{align*}
where the reflections $R_A$ and $R_B$ are defined by $R_A := 2 \pi_A - \id$ and $R_B := 2 \pi_B - \id$, respectively.
In the case of the data-driven problem, 
we have the two alternatives $A \coloneqq \DD_k$ and $B \coloneqq \EE_k$ or $A \coloneqq \EE_k$ and $B \coloneqq \DD_k$, that is
\begin{align*}
	y_{n+1} = T_{\EE_k,\DD_k}(y_n) \tag{DR1} \label{eq:DR1}
\end{align*}
or, respectively,
\begin{align*}
	y_{n+1} = T_{\DD_k,\EE_k}(y_n). \tag{DR2} \label{eq:DR2}
\end{align*}
Since in general $\EE_k \cap \DD_k = \emptyset$, we can neither expect the existence of 
a fixed point of $T_{\EE_k,\DD_k}$ nor $T_{\DD_k,\EE_k}$, cf.\ \cite[Proposition 9]{ACT20}. 
Consequently, the iterations in \eqref{eq:DR1} and \eqref{eq:DR2}, respectively, will in general not converge, when applied to 
the data-driven problem. Though convergence of the iterations in \eqref{eq:P} and \eqref{eq:PS} can be expected neither, 
we observed their convergence to a fixed point in our numerical computations. In contrast to this, the Douglas-Rachford iterations 
did frequently not converge to a fixed point and therefore, we let the algorithm terminate, 
if it does not achieve an improvement of the objective value for a fixed number of iterations.
Note that the iterates of both variants of the Douglas-Rachford algorithm do not necessarily fulfill $y_{n+1} \in \DD_k$ 
and hence, in order to evaluate the objective of \eqref{eq:Pk2} in the form \eqref{eq:Pkreform}, we 
additionally project the current iterate onto $\DD_k$ after each iteration to evaluate the objective value.


\subsection{Quadratic assignment and local search}\label{sec:qsap}

In order to design an alternative strategy for the solution of \eqref{eq:Pk2}, 
we rewrite the discretized data-driven problem as a \emph{quadratic semi-assignment problem}, 
where we assign the measured data to the elements of the finite element grid 
such that the distance to the corresponding projection onto the equilibrium set becomes minimal. 
For this purpose, let $\TT_{h_k} = \{T_1, \dots , T_{l}\}$, $l = l(h_k) \in \N$, be the considered triangulation of $\Omega$ 
and let $\DD_k$ be defined as in \eqref{eq:DDexk} with $\DD_{k}^{\loc} \coloneqq \{ \hat{y}_1 , \dots , \hat{y}_m \}$, $m\in \N$. 
Then \eqref{eq:Pk2} can be rewritten as
\begin{equation}\tag{QSAP}\label{eq:qsap}
    \eqref{eq:Pk2} 
    \quad \Longleftrightarrow \quad
    \left\{\quad
    \begin{aligned}
    	\min \quad & \frac{1}{2} \| \pi_{\EE_k}(y) - y \|_Z^2 \\
	    \mathrm{s. t.} \quad & y |_{T_i}  = \sum_{j=1}^{m} x_{i,j} \hat{y}_j \quad \forall \, i = 1, \dots, n \\
	    & \sum_{j=1}^{m}x_{i,j} = 1 \quad \forall \, i = 1,...,l \\
	    & x_{i,j} \in \{0,1\} \quad \forall \, i = 1,...,l, \quad j = 1,...,m
    \end{aligned}
    \quad\right\}
\end{equation}
with $x = (x_{1,1}, \dots, x_{1,m}, x_{2,1}, \dots, x_{2,m}, \dots, x_{l,1}, \dots, x_{n,m})^T \in \R^{l m}$, or equivalently
\begin{equation*}
    \eqref{eq:qsap}    
    \quad \Longleftrightarrow \quad
    \left\{\quad
    \begin{aligned}
    	\min \quad & x^T A x + b^Tx + c \\
	    \mathrm{s. t.} \quad & \sum_{j=1}^{m}x_{i,j} = 1 \quad \forall \, i = 1,...,l \\
	    & x_{i,j} \in \{0,1\} \quad \forall \, i = 1,...,l, \quad j = 1,...,m
    \end{aligned}
    \right.
\end{equation*}
with a symmetric and positive semidefinite matrix $A \in \R^{l m \times l m}$, $b \in \R^{l m}$, 
and $c \in \R$ arising from the affine solution operator to the saddle point system \eqref{eq:KKTh}, 
the mass matrices for the scalar products in $Z$, and a matrix containing the elements of $\DD_k^{\loc}$.
Note that it is allowed to assign the same measuring point to more than one element and that the objective function is quadratic. 

\begin{remark}\label{rem:qsap}
    Since quadratic semi-assignment problems are NP-hard \cite{Sahni1976Pcompleteapprox}, 
    the reformulation as \eqref{eq:qsap} shows that it is in general not possible to solve \eqref{eq:Pk2}
    efficiently (provided that $NP \neq P$). 
    This makes it practically impossible to solve the discretized data-driven problem exactly, 
    when we are faced with a big amount of measured data and a fine finite element grid.
\end{remark}

In view of the above remark, we resort to a heuristic for the solution of \eqref{eq:qsap}.
Such a heuristic method is the local search, that is we start with an arbitrary initial assignment 
and change it on only one element of the triangulation while the assignment to the remaining elements is fixed. 
If the change leads to an improvement, the current solution is updated accordingly
and one moves on to the next element in the triangulation, where the same procedure is applied.
To test every point in local data set $\DD_k^\loc$ for the respective element of the triangulation rapidly becomes 
too costly, even for moderate finite element meshes, since one needs to evaluate the objective in \eqref{eq:qsap} every time, 
which in turn requires the solution of the saddle point problem \eqref{eq:KKTh}.
In order to reduce the effort, one can restrict to the $K$ nearest neighbors of the data point
which is currently assigned to the considered element. 
But still, for reasonable finite element discretizations, the effort of the method easily becomes too high. 
A possible remedy is to employ a reduced model-order approach such as proper orthogonal decomposition (POD) 
for the saddle point problem in \eqref{eq:KKTh}, cf.\ e.g.\ \cite{Kahlbacher2007GalerkinPOD,Ullmann2016PODadaptive} 
and the references therein.
Since the right hand side only changes in one single element from one iteration of the local search to the next, 
only a moderate number of snapshots are needed to achieve a sufficient accuracy of the POD basis.
The overall method is sketched in Algorithm~\ref{alg:local_search_POD}.
\begin{algorithm}[h!]
	\caption{Local Search with POD} \label{alg:local_search_POD}
	\begin{algorithmic}[1]
		\State Choose triangulation $\TT$ of $\Omega$
		\State Choose $\varepsilon_1, \varepsilon_2, \varepsilon_3 > 0$ and $K \in \N$
		\State Compute initial POD-basis
		\State Choose initial solution $y$ with objective value $v :=  \frac{1}{2} \| \pi_{\EE_k}(y) - y \|_Z^2 $
		\State Set $\bar{v} \coloneqq v+1$
		\While{$v \neq \bar{v}$}
		\State Update: $\bar{v} \gets v$
		\ForAll{$\tilde{T} \in \TT$}
		\State Find $K$ nearest neighbors $y_1, \dots,y_K \in \DD_k^{\loc}$ of $y\vert_{\tilde{T}}$
		\For{$j=1,...,K$}
		\State Define $\tilde{y}$ on each $T \in \TT$ by
		\begin{align*}
			\tilde{y}\vert_{T} := \begin{cases}
				y\vert_{T} & \text{if } T \neq \tilde{T} \\
				y_j, & \text{if } T = \tilde{T}
			\end{cases}
		\end{align*}
		\State Solve POD-model to approximate $\pi_{\EE_k}(\tilde{y})$ by $z_a$
		\State Set $v_a \coloneqq \frac{1}{2} \| \tilde{y} - z_a \|^2_Z$
		\If{$\frac{\vert v_a - v \vert }{\vert v \vert}> \varepsilon_1$}
		\State Compute exact solution $z_e \coloneqq \pi_{\EE_k} (\tilde{y})$ and set
		\begin{align*}
			v_e \coloneqq \frac{1}{2} \| \tilde{y} - z_e \|^2_Z
		\end{align*}
		\If{$ \frac{\vert v_a - v_e \vert}{\vert v_e \vert} > \varepsilon_2$}
		\State Add $z_e$ as snapshot to compute new POD-basis
		\EndIf
		\If{$v_e < v$}
		\State Update: $y \gets \tilde{y}$, $v \gets v_e$
		\State Add $z_e$ as snapshot to compute new POD-basis
		\EndIf
		\EndIf
		\If{$v_a < (1+\varepsilon_3) v$}
		\State Compute exact solution $z_e \coloneqq \pi_{\EE_k} (\tilde{y})$ and set
		\begin{align*}
			v_e \coloneqq \frac{1}{2} \| \tilde{y} - z_e \|^2
		\end{align*}
		\If{$v_e < v$}
		\State Update: $y \gets \tilde{y}$, $v \gets v_e$
		\State Add $z_e$ as snapshot to compute new POD-basis
		\EndIf
		\EndIf
		\EndFor
		\EndFor
		\EndWhile
	\end{algorithmic}
\end{algorithm}
Note that we only accept an improvement on the objective if it is computed with the exact finite element model.
A crucial issue for the performance of Algorithm~\ref{alg:local_search_POD} is the choice of the initial solution.
One possibility is to initialize the local search by a solution obtained with one of the projection algorithms from Section~\ref{sec:prox}.
This allows to escape from fixed points that are not optimal or just local minimizers.
A second option is to solve \eqref{eq:qsap} exactly with a small selection of measured data points on a very coarse finite element mesh. 
Note that, since \eqref{eq:qsap} is NP-hard, the problem size has to be fairly small to allow for an exact algorithm. 
We employ the algorithm from \cite{Buchheim2018Quadratic} on a coarse grid with only a few measured data points
and project its solution to the finer mesh to initialize Algorithm~\ref{alg:local_search_POD}.


\section{Numerical experiments}\label{sec:numerics}

In our numerical tests, we consider the domain $\Omega = (0,1)^2$ such that Assumption~\ref{assu:domreg}\ref{it:domreg} 
is fulfilled. For the finite element discretization,
we use an exact triangulation of Friedrich-Keller type and choose the following function spaces
\begin{equation*}
	Q_h = \RR \TT_0(\TT_h), \quad
	V_h = \PP_0(\TT_h),
\end{equation*}
and
\begin{equation}\label{eq:Uh}
	U_h = \{ u \in C(\overline{\Omega}) \cap H^1_0(\Omega) \colon u|_T \in \PP_1(T) \; \forall \, T \in \TT_h \}.
\end{equation}
As seen in Proposition~\ref{prop:RT}, these finite element spaces satisfy Assumption~\ref{assu:discret}.
 
We test all previously introduced algorithms with the following parameters and settings: 
For the projection-based fixed point algorithms \eqref{eq:P}, \eqref{eq:PS}, \eqref{eq:DR1}, and \eqref{eq:DR2}, we use 
$y_0 = 0$ as starting point.
If not stated otherwise, we choose the initial step size $\gamma = 1.4$ for \eqref{eq:PS} 
and reduced $\gamma$ by the factor $0.9$ whenever the algorithm circled between two iterates. 
This choice is a compromise of accuracy and running time motivated by the observations in Table~\ref{tab:gamma} below.
As mentioned above, the Douglas-Rachford methods in \eqref{eq:DR1} and \eqref{eq:DR2} do frequently not converge to a fixed point 
and hence, we let the algorithms terminate after $50$ iterations without an improvement of the objective value. 

We also test two variants of Algorithm~\ref{alg:local_search_POD}. 
In the first one, we initialized the algorithm with the best out of ten solutions of the projection algorithm \eqref{eq:PS} 
with random starting points with components in $[-4,4]^4$ and used the remaining for the computation of the initial POD basis. 
In the second one, we initialized the algorithm with the exact solution on a coarse grid consisting of eight triangles 
with mesh size $\frac{1}{\sqrt{2}}$ and $16$ selected data points of our measured data set. 
As indicated above, we computed this solution by means of the exact algorithm from \cite{Buchheim2018Quadratic}. 
In both cases, the parameters of Algorithm~\ref{alg:local_search_POD} 
are set to $\varepsilon_1 = 0.002$, $\varepsilon_2 = 0.001$, $\varepsilon_3 = 0.01$, and 
$K= 20$.
These choices are motivated by numerical experiments based on the linear material law from Section~\ref{sec:fourier}.

\subsection{Fourier's law}\label{sec:fourier}

For the first numerical tests, we considered Fourier's law where the material law is linear. 
To be more precise, we set $\kappa : \R^2 \to \R^2$, $\kappa(w) \coloneqq w$ such that \eqref{SDP} simply becomes Poisson's equation, 
i.e., $- \Delta u = f \text{ in } \Omega$. 
Moreover, the right hand side is set to $f(x) = 2 \pi^2 \sin(\pi x_1) \sin(\pi x_2)$, which clearly fulfills the 
regularity condition in Assumption~\ref{assu:domreg}.
Accordingly, the exact solution reads $u(x) = \sin(\pi x_1) \sin(\pi x_2)$.
For the local material data set, we sampled $11025$ evenly distributed data points $(r,w)$ 
fulfilling $r = \kappa(w)$ within $[-4,4]^4$.

To compare the algorithms, we list the returned objective values and the needed iterations and wall clock times
in Table~\ref{tab:fourier}. Moreover, we compute the relative distance of the exact solution to the outcome of the respective 
algorithm measured in the $L^2$-norm and in the $H^1_0$-seminorm.
For this purpose, denote the result of the respective algorithm by $\bar y$ and set 
$(\bar\bq_{h_k}, \nabla \bar u_{h_k}) := \pi_{\EE_k}(\bar y)$. We then compute 
\begin{equation}\label{eq:errors}
    \err_{L^2} := \frac{\| u - \bar u_{h_k} \|_{L^2(\Omega)}}{\| u \|_{L^2(\Omega)}} 
    \quad \text{and} \quad 
    \err_{H^1_0} := \frac{\| \nabla u - \nabla \bar u_{h_k} \|_{L^2(\Omega;\R^d)}}{\| \nabla u \|_{L^2(\Omega;\R^d)}}.
\end{equation}
\begin{table}[h!]
	\resizebox{\textwidth}{!}{%
		\begin{tabular}{|c|c|c|c|c|c|}
			\hline
			Algorithm & Objective value & $\err_{L^2}$ & $\err_{H^1_0}$ & Iterations & Time \\
			\hline
			Projection \eqref{eq:P} & 1.281e-02 & 1.731e-02 & 7.973e-02 & 10 & 0.477 \\
			\hline
			Projection with stepsize \eqref{eq:PS} & 1.248e-02 & 8.869e-03 & 7.895e-02 & 17 & 0.554  \\
			\hline
			Douglas-Rachford \eqref{eq:DR1} & 1.305e-02 & 7.878e-03 & 7.891e-02 & 99 & 9.137 \\
			\hline
			Douglas-Rachford \eqref{eq:DR2} & 1.299e-02 & 8.292e-03 & 7.870e-02 & 52 & 3.846  \\
			\hline
			Algorithm \ref{alg:local_search_POD} with initialization by \eqref{eq:PS} & 1.247e-02 & 8.800e-03 & 7.886e-02 & 4 & 181.242 \\
			\hline
			Algorithm \ref{alg:local_search_POD} with exact initialization & 1.248e-02 & 8.800e-03 & 7.887e-02 & 21 & 4351.329 \\
			\hline
	\end{tabular}}
	\caption{Fourier's law with $|\DD_k^\loc| = 11025$ and $h_k := \sqrt2/20$.} \label{tab:fourier}
\end{table}
As one can see, the returned objective values of the algorithms and the distances to the exact solution are all similar to each other. 
It is the computing time that makes the most significant difference. Since the local search in Algorithm~\ref{alg:local_search_POD} 
considers each element of the triangulation separately, there is much computational effort for only little improvements.
However, the algorithm involves several parameters, beside the tolerances $\varepsilon_1$, $\varepsilon_2$, and $\varepsilon_3$ 
for the update of POD-basis and the number $K$ of nearest neighbors, in addition the size of the POD-basis and 
the choice of the initial point. Moreover, the exact algorithm from \cite{Buchheim2018Quadratic} can recursively be applied 
within the local search by fixing the assignment on large parts of the domain, while one applies the exact algorithm on a small 
amount of elements with only a few selected measured data points.
A comprehensive numerical study of Algorithm~\ref{alg:local_search_POD} including the adjustment of all these parameters
however goes beyond the focus of this work and gives rise to future research.
But, in view of substantial differences w.r.t.\ the computing time, it is at least doubtful, 
if the local search could ever be competitive compared to the projection-based methods.
For this reason, we left out Algorithm~\ref{alg:local_search_POD} for the following numerical study of a non-linear material law.


\subsection{A non-linear material law}

In our second numerical test, we considered the non-linear material law $\kappa : \R^2 \to \R^2$ defined by 
\begin{equation}\label{eq:kappa}
    \kappa (\bw) \coloneqq (2 \tan^{-1}(\|\bw\|^2-1)+0.5 \pi+2) \bw.
\end{equation}
Note that $\kappa$ is strongly monotone and globally Lipschitz so that, according to the Browder and Minty theorem,
there is a unique solution in $H^1_0(\Omega)$ of \eqref{SDP} for every right hand side in $H^{-1}(\Omega)$.
This time the right hand side is set to $f(x) = - \div (\kappa(\nabla u(x
)))$ with $u(x) = \sin(\pi x_1) \sin(\pi x_2)$, so that the exact solution of \eqref{SDP}
is given by $u$.

For the numerical computations, 
we randomly sampled $|\DD_k^\loc| = 1000$, $10000$, $50000$, $100000$ uniformly distributed data-points $(\br, \bw)$ 
fulfilling $\br = \kappa(\bw)$ with $\bw \in [-4,4]^2$. Additionally, a uniformly distributed noise $\bs_i \in [-\bar{s},\bar{s}]^4$, 
$i = 1,...,|\DD_k^\loc|$,
with $\bar{s} = 0, 0.1$, $0.01$, $0.001$ is randomly added to each data point. For the finite element discretization, 
we considered the mesh sizes $h_1 = \frac{\sqrt{2}}{50}$, $h_2 = \frac{\sqrt{2}}{100}$, and $h_3 = \frac{\sqrt{2}}{200}$. 

First we investigated the choice of the initial step size in \eqref{eq:PS}. Recall that the step size is reduced, whenever 
the algorithm circles between two iterates. The results are shown in Table~\ref{tab:gamma} and 
illustrate that the initial choice $\gamma = 1.4$ is favorable in terms of the value of the objective, but let the 
computing time increase in comparison to the original projection algorithm \ref{eq:P}, which corresponds to the case $\gamma = 1$.
The course of the objective values for $\gamma = 1.4$ produced by \eqref{eq:PS} is illustrated in Figure~\ref{fig:gamma}.
\begin{minipage}{0.5\textwidth}
	
\end{minipage}
\begin{table}[h!]
	\begin{tabular}{|c|c|c|c|c|c|c|}
		\hline
		\multicolumn{3}{|c|}{$|\DD_k^\loc| = 11025$, $\bar{s} = 0$} && \multicolumn{3}{c|}{$|\DD_k^\loc| = 5000$, $\bar{s} = 0.1$} \\
		\hline
		Init. $\gamma$ & Iterations & Objective value && Init. $\gamma$ & Iterations & Objective value \\
		\hline
		0.8 & 14 & 3.907e-03 && 0.8 & 12 & 9.839e-02 \\
		0.9 & 12 & 3.289e-03 && 0.9 & 12 & 8.913e-02 \\
		1.0 & 11 & 2.899e-03 && 1.0 & 11 & 7.841e-02 \\
		1.1 & 13 & 2.700e-03 && 1.1 & 12 & 7.420e-02 \\
		1.2 & 13 & 2.606e-03 && 1.2 & 23 & 6.414e-02 \\
		1.3 & 18 & 2.573e-03 && 1.3 & 47 & 6.155e-02 \\
		1.4 & 36 & 2.558e-03 && 1.4 & 93 & 6.034e-02 \\
		1.5 & 40 & 2.564e-03 && 1.5 & 112 & 6.208e-02 \\
		1.6 & 49 & 2.566e-03 && 1.6 & 137 & 6.504e-02 \\
		1.7 & 57 & 2.562e-03 && 1.7 & 168 & 6.816e-02 \\
		1.8 & 58 & 2.567e-03 && 1.8 & 218 & 7.102e-02 \\
		1.9 & 85 & 2.564e-03 && 1.9 & 317 & 7.054e-02 \\
		2.0 & 88 & 2.563e-03 && 2.0 & 651 & 6.944e-02 \\
		2.1 & 560 & 2.562e-03 && 2.1 & 763 & 8.738e-02 \\
		2.2 & 566 & 2.567e-03 && 2.2 & 1052 & 6.451e-02 \\
		\hline
	\end{tabular}
	\caption{Testing the projection algorithm with step size \eqref{eq:PS} with different initial values for $\gamma$ with $h = \sqrt{2}/50$.} \label{tab:gamma}
\end{table}

\begin{figure}[H]
	\begin{minipage}[t]{0.49\textwidth}
		\includegraphics[width=\textwidth]{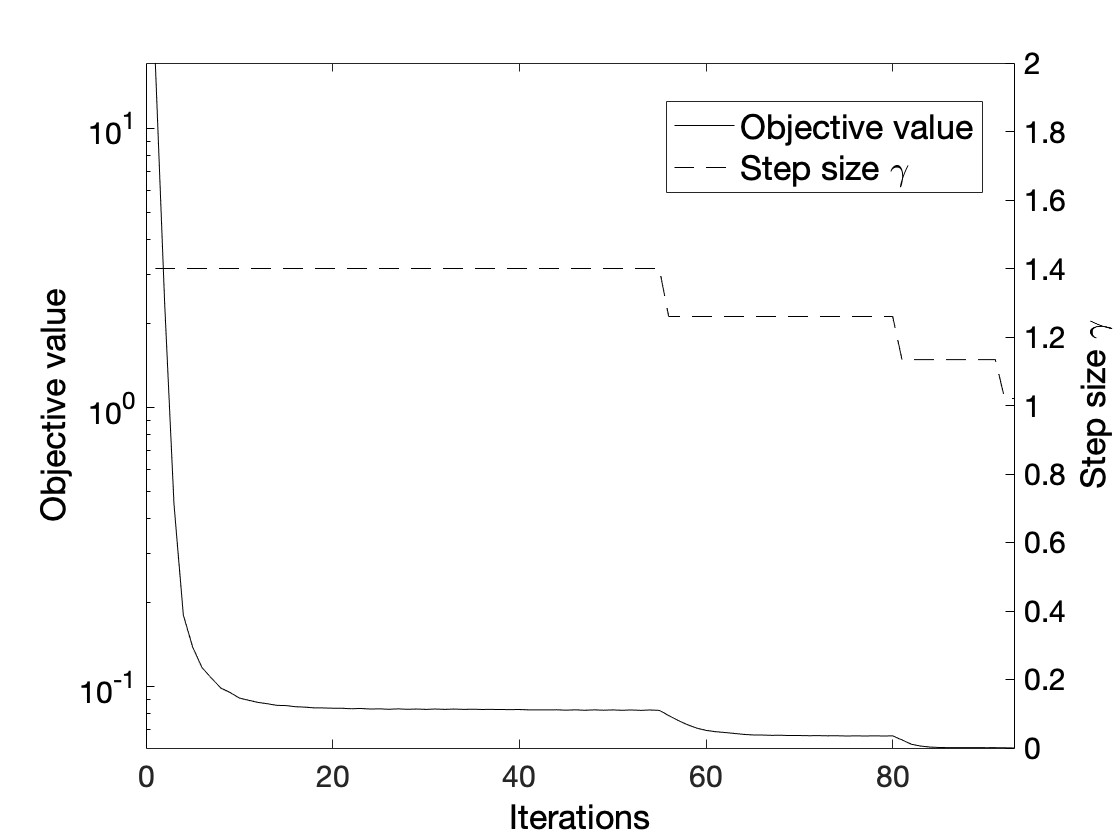}
	\end{minipage}
	\begin{minipage}[t]{0.49\textwidth}
		\includegraphics[width=\textwidth]{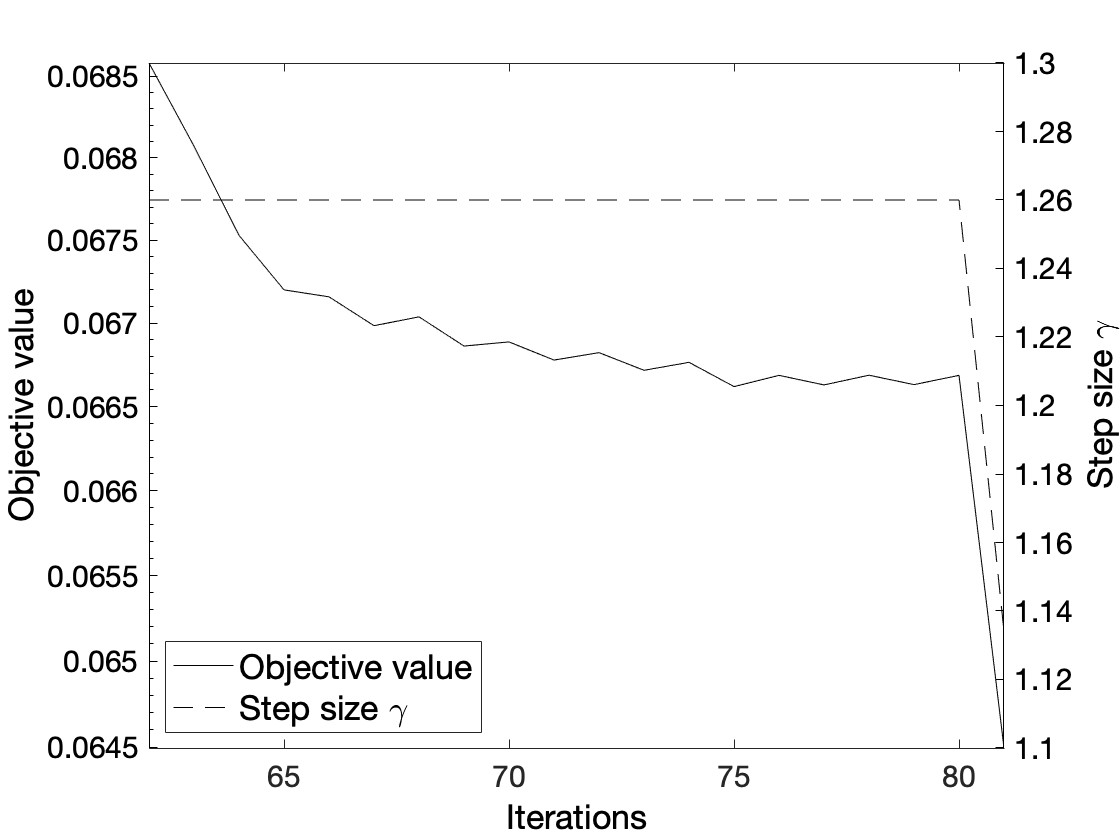}
	\end{minipage}
	\caption{Objective values produced by \eqref{eq:PS} with initial value $\gamma = 1.4$, 
	$h = \sqrt{2}/50$, $|\DD_k^\loc| = 5000$, and $\bar{s} = 0.1$.} \label{fig:gamma}
\end{figure}

The errors from \eqref{eq:errors} as well as the objective value and the computing times for this example 
for the different data sets $\DD_{k}^\loc$, various noise levels $\bar s$, and different mesh sizes are shown in Table~\ref{tab:NL}. 
In addition, the number of iterations and the computing times for the various objectives are listed.
\begin{table}[!p]
	\centering
	\rotatebox{90}{
		\begin{minipage}[c]{\textheight}
			\resizebox{\textwidth}{!}{%
				\begin{tabular}{|c|c|c|c|c|c|c|c|c|c|c|c|c|c|c|c|c|c|c|c|c|c|c|}
					\hline
					Grid & \multicolumn{2}{  c|}{Data set} & \multicolumn{4}{c|}{Objective value} & \multicolumn{4}{c|}{$\err_{L^2}$} & \multicolumn{4}{c|}{$\err_{H^1_0}$} & \multicolumn{4}{c|}{Iterations} & \multicolumn{4}{c|}{Time} \\
					\hline
					$N$ & $|\DD_k^\loc|$ & $\bar{s}$ & \eqref{eq:P} & \eqref{eq:PS} & \eqref{eq:DR1} & \eqref{eq:DR2} &\eqref{eq:P} & \eqref{eq:PS} & \eqref{eq:DR1} & \eqref{eq:DR2} & \eqref{eq:P} & \eqref{eq:PS} & \eqref{eq:DR1} & \eqref{eq:DR2} & \eqref{eq:P} & \eqref{eq:PS} & \eqref{eq:DR1} & \eqref{eq:DR2} & \eqref{eq:P} & \eqref{eq:PS} & \eqref{eq:DR1} & \eqref{eq:DR2} \\
					\hline
					50 & 5000   & 0.1   & 7.841e-02 & 6.034e-02 & 2.155e-01 & 1.499e-01 & 5.728e-03 & 2.199e-03 & 4.759e-03 & 2.353e-03 & 4.075e-02 & 3.930e-02 & 4.480e-02 & 4.215e-02 & 11 & 93 & 76 & 57 & 2.488 & 8.027 & 12.060 & 9.741 \\
					&        & 0.01  & 7.617e-02 & 5.916e-02 & 1.576e-01 & 1.377e-01 & 4.729e-03 & 2.590e-03 & 2.388e-03 & 2.395e-03 & 3.461e-02 & 3.337e-02 & 3.751e-02 & 3.648e-02 & 15 & 86 & 86 & 59 & 2.755 & 7.571 & 14.135 & 10.058 \\
					&        & 0.001 & 7.438e-02 & 5.814e-02 & 1.510e-01 & 1.313e-01 & 4.542e-03 & 2.342e-03 & 2.886e-03 & 2.164e-03 & 3.441e-02 & 3.333e-02 & 3.783e-02 & 3.620e-02 & 10 & 85 & 62 & 60 & 2.418 & 7.504 & 10.407 & 10.228 \\
					&        & 0     & 7.669e-02 & 5.857e-02 & 1.547e-01 & 1.310e-01 & 4.799e-03 & 2.337e-03 & 3.014e-03 & 1.979e-03 & 3.470e-02 & 3.331e-02 & 3.808e-02 & 3.637e-02 & 11 & 87 & 62 & 61 & 2.506 & 7.639 & 10.413 & 10.381 \\
					\hline
					50 & 10000  & 0.1   & 5.282e-02 & 4.429e-02 & 1.513e-01 & 9.142e-02 & 7.521e-03 & 3.387e-03 & 2.713e-03 & 3.493e-03 & 3.907e-02 & 3.811e-02 & 4.158e-02 & 3.924e-02 & 16 & 103 & 70 & 53 & 2.908 & 9.324 & 12.276 & 9.782 \\
					&        & 0.01  & 5.245e-02 & 4.327e-02 & 9.670e-02 & 8.442e-02 & 3.608e-03 & 2.297e-03 & 2.493e-03 & 1.884e-03 & 3.322e-02 & 3.263e-02 & 3.542e-02 & 3.420e-02 & 13 & 93 & 62 & 58 & 2.693 & 8.573 & 11.538 & 10.891 \\
					&        & 0.001 & 5.228e-02 & 4.338e-02 & 9.398e-02 & 8.446e-02 & 3.468e-03 & 2.317e-03 & 2.791e-03 & 2.186e-03 & 3.321e-02 & 3.252e-02 & 3.519e-02 & 3.441e-02 & 15 & 75 & 62 & 60 & 2.833 & 7.300 & 11.602 & 11.267 \\
					&        & 0     & 5.332e-02 & 4.429e-02 & 9.332e-02 & 8.526e-02 & 3.751e-03 & 2.277e-03 & 1.385e-03 & 2.028e-03 & 3.316e-02 & 3.256e-02 & 3.511e-02 & 3.441e-02 & 14 & 123 & 85 & 63 & 2.763 & 10.712 & 15.754 & 11.796 \\
					\hline
					50 & 50000  & 0.1   & 3.310e-02 & 2.998e-02 & 8.089e-02 & 4.559e-02 & 8.189e-03 & 2.355e-03 & 3.023e-03 & 1.492e-03 & 3.847e-02 & 3.792e-02 & 3.686e-02 & 3.639e-02 & 11 & 78 & 157 & 53 & 3.119 & 11.476 & 44.129 & 15.502 \\
					&        & 0.01  & 3.222e-02 & 2.966e-02 & 4.670e-02 & 3.830e-02 & 2.407e-03 & 1.999e-03 & 2.625e-03 & 1.948e-03 & 3.214e-02 & 3.192e-02 & 3.292e-02 & 3.220e-02 & 11 & 85 & 62 & 58 & 3.125 & 12.414 & 21.290 & 18.862 \\
					&        & 0.001 & 3.256e-02 & 2.987e-02 & 3.976e-02 & 3.804e-02 & 2.372e-03 & 2.048e-03 & 2.454e-03 & 1.635e-03 & 3.210e-02 & 3.187e-02 & 3.262e-02 & 3.230e-02 & 19 & 88 & 62 & 63 & 4.075 & 12.700 & 21.534 & 20.895 \\
					&        & 0     & 3.227e-02 & 2.982e-02 & 3.933e-02 & 3.777e-02 & 2.417e-03 & 2.118e-03 & 2.380e-03 & 1.799e-03 & 3.210e-02 & 3.186e-02 & 3.259e-02 & 3.210e-02 & 16 & 84 & 62 & 59 & 3.807 & 12.282 & 21.654 & 19.565 \\
					\hline
					50 & 100000 & 0.1   & 2.932e-02 & 2.738e-02 & 6.180e-02 & 3.777e-02 & 9.482e-03 & 1.854e-03 & 1.779e-03 & 1.671e-03 & 3.837e-02 & 3.772e-02 & 3.606e-02 & 3.578e-02 & 12 & 91 & 170 & 53 & \textbf{4.223} & \textbf{20.768} & \textbf{77.888} & 24.637 \\
					&        & 0.01  & 2.983e-02 & 2.796e-02 & 4.141e-02 & 3.263e-02 & 2.375e-03 & 2.111e-03 & 2.168e-03 & 1.884e-03 & 3.204e-02 & 3.184e-02 & 3.265e-02 & 3.191e-02 & 15 & 92 & 62 & 58 & 4.831 & 20.908 & 35.731 & 31.038 \\
					&        & 0.001 & 2.967e-02 & 2.802e-02 & 3.341e-02 & 3.185e-02 & 2.388e-03 & 2.091e-03 & 2.375e-03 & 1.234e-03 & 3.195e-02 & 3.175e-02 & 3.235e-02 & 3.186e-02 & 24 & 109 & 62 & 83 & 6.690 & 24.436 & 36.423 & 50.621 \\
					&        & 0     & 2.982e-02 & 2.794e-02 & 3.284e-02 & 3.152e-02 & 2.311e-03 & 2.058e-03 & 2.454e-03 & \textit{1.598e-03} & 3.196e-02 & \textit{3.174e-02} & 3.232e-02 & 3.180e-02 & 16 & 98 & 62 & 59 & 5.068 & 22.188 & 36.668 & 32.388 \\
					\hline
					100 & 5000   & 0.1   & \textbf{5.375e-02} & \textbf{3.622e-02} & 1.672e-01 & 1.139e-01 & 5.097e-03 & 2.123e-03 & 2.942e-03 & 2.211e-03 & 3.218e-02 & 2.966e-02 & 3.365e-02 & 3.133e-02 & 16 & 101 & 124 & 60 & 18.214 & 44.105 & 88.854 & 50.289 \\
					&        & 0.01  & 5.250e-02 & 3.598e-02 & 1.327e-01 & 1.067e-01 & 4.205e-03 & 1.835e-03 & 2.482e-03 & 1.371e-03 & 2.192e-02 & 1.910e-02 & 2.616e-02 & 2.369e-02 & 18 & 109 & 151 & 59 & 18.748 & 46.334 & 108.851 & 49.967 \\
					&        & 0.001 & 4.894e-02 & 3.511e-02 & 1.250e-01 & 1.022e-01 & 3.442e-03 & 1.504e-03 & 1.872e-03 & 1.128e-03 & 2.133e-02 & 1.900e-02 & 2.602e-02 & 2.320e-02 & 14 & 98 & 62 & 60 & 17.622 & 43.469 & 52.414 & 50.571 \\
					&        & 0     & 5.105e-02 & 3.507e-02 & 1.248e-01 & 1.051e-01 & 3.544e-03 & 1.638e-03 & 2.017e-03 & 1.223e-03 & 2.174e-02 & 1.894e-02 & 2.584e-02 & 2.356e-02 & 15 & 95 & 186 & 60 & 17.949 & 42.387 & 132.577 & 50.536 \\
					\hline
					100 & 10000  & 0.1   & 3.336e-02 & 2.404e-02 & 1.003e-01 & 6.893e-02 & 7.693e-03 & 2.599e-03 & 2.138e-03 & 2.387e-03 & 3.017e-02 & 2.796e-02 & 2.918e-02 & 2.826e-02 & 14 & 96 & 153 & 58 & 17.698 & 43.270 & 108.273 & 49.992 \\
					&        & 0.01  & 3.290e-02 & 2.324e-02 & 7.563e-02 & 6.133e-02 & 2.586e-03 & 1.365e-03 & 1.614e-03 & 1.124e-03 & 1.937e-02 & 1.793e-02 & 2.261e-02 & 2.064e-02 & 13 & 104 & 62 & 59 & 17.472 & 45.548 & 54.508 & 51.185 \\
					&        & 0.001 & 3.266e-02 & 2.346e-02 & 7.415e-02 & 6.071e-02 & 2.492e-03 & 1.275e-03 & 1.958e-03 & 1.030e-03 & 1.926e-02 & 1.783e-02 & 2.262e-02 & 2.053e-02 & 18 & 121 & 62 & 60 & 18.898 & 50.560 & 54.112 & 51.783 \\
					&        & 0     & 3.287e-02 & 2.385e-02 & 7.463e-02 & 6.186e-02 & 2.501e-03 & 1.477e-03 & 2.031e-03 & 1.221e-03 & 1.924e-02 & 1.790e-02 & 2.255e-02 & 2.057e-02 & 17 & 109 & 62 & 60 & 18.582 & 47.091 & 54.112 & 51.816 \\
					\hline
					100 & 50000  & 0.1   & 1.521e-02 & 1.228e-02 & 3.220e-02 & 2.723e-02 & 8.192e-03 & 1.810e-03 & \cellcolor[gray]{.9} 8.163e-04 & 1.543e-03 & 2.846e-02 & 2.758e-02 & 2.466e-02 & 2.517e-02 & 16 & 100 & 332 & 53 & 19.213 & 49.814 & 251.952 & 52.959 \\
					&        & 0.01  & 1.403e-02 & 1.157e-02 & 2.808e-02 & 2.008e-02 & 1.575e-03 & 1.216e-03 & 6.625e-04 & 9.816e-04 & 1.720e-02 & 1.672e-02 & 1.855e-02 & 1.728e-02 & 22 & 110 & 85 & 59 & 21.289 & 53.372 & 89.297 & 60.222 \\
					&        & 0.001 & 1.425e-02 & 1.156e-02 & 2.165e-02 & 1.971e-02 & 1.595e-03 & 1.252e-03 & 1.629e-03 & 9.382e-04 & 1.714e-02 & 1.659e-02 & 1.812e-02 & 1.711e-02 & 17 & 114 & 62 & 103 & 19.565 & 54.822 & 68.245 & 103.937 \\
					&        & 0     & 1.412e-02 & 1.157e-02 & 2.103e-02 & 1.945e-02 & 1.524e-03 & 1.259e-03 & \cellcolor[gray]{.9} 1.587e-03 & 4.632e-04 & 1.707e-02 & 1.661e-02 & 1.803e-02 & 1.708e-02 & 23 & 134 & 62 & 111 & 21.630 & 61.853 & 68.280 & 113.192 \\
					\hline
					100 & 100000 & 0.1   & 1.183e-02 & 9.950e-03 & 2.087e-02 & 1.870e-02 & \cellcolor[gray]{.9} 1.002e-02 & 1.104e-03 & 7.591e-04 & 7.054e-04 & 2.840e-02 & 2.710e-02 & 2.393e-02 & 2.360e-02 & 15 & 136 & 295 & 163 & \textbf{20.175} & 74.482 & 276.409 & \textbf{157.545} \\
					&        & 0.01  & 1.157e-02 & 9.637e-03 & 2.180e-02 & 1.416e-02 & 1.543e-03 & 1.246e-03 & 5.781e-04 & 8.339e-04 & 1.693e-02 & 1.648e-02 & 1.787e-02 & 1.669e-02 & 23 & 156 & 85 & 59 & 23.568 & 83.122 & 116.541 & 73.134 \\
					&        & 0.001 & 1.147e-02 & 9.631e-03 & 1.470e-02 & 1.335e-02 & 1.572e-03 & 1.252e-03 & 1.640e-03 & 9.232e-04 & 1.681e-02 & 1.636e-02 & 1.749e-02 & 1.657e-02 & 36 & 141 & 62 & 60 & 29.353 & 76.426 & 88.675 & 76.115 \\
					&        & 0     & 1.146e-02 & 9.607e-03 & 1.443e-02 & 1.317e-02 & \cellcolor[gray]{.9} 1.505e-03 & 1.262e-03 & 1.614e-03 & \textit{3.616e-04} & 1.680e-02 & \textit{1.637e-02} & 1.743e-02 & 1.652e-02 & 18 & 197 & 62 & 114 & 21.525 & 100.915 & 88.232 & 156.888 \\
					\hline
					200 & 5000   & 0.1   & 4.101e-02 & \textbf{2.840e-02} & \textbf{1.347e-01} & \textbf{1.006e-01} & 5.324e-03 & 2.168e-03 & 1.941e-03 & 1.975e-03 & 3.083e-02 & 2.772e-02 & 2.977e-02 & 2.851e-02 & 17 & 93 & 211 & 61 & 249.322 & 370.015 & 813.302 & 411.318 \\
					\hhline{|~|~|~|~|-|~|~|~|~|~|~|~|~|~|~|~|~|~|~|~|~|~|~|}
					&        & 0.01  & 4.000e-02 & \vline \cellcolor[gray]{1.0} 2.778e-02 \vline & 1.166e-01 & 9.346e-02 & 3.989e-03 & 1.582e-03 & 1.083e-03 & 1.073e-03 & \textbf{1.740e-0}2 & \textbf{1.347e-02} & 2.107e-02 & 1.888e-02 & 25 & 115 & 236 & 61 & 260.467 & 398.098 & 898.825 & 411.689 \\
                  \hhline{|~|~|~|~|-|~|~|~|~|~|~|~|~|~|~|~|~|~|~|~|~|~|~|}
					&        & 0.001 & 3.730e-02 & 2.679e-02 & 1.084e-01 & 8.936e-02 & 3.152e-03 & 1.284e-03 & 7.409e-04 & 1.035e-03 & 1.663e-02 & 1.316e-02 & 2.055e-02 & 1.839e-02 & 31 & 92 & 290 & 59 & 268.484 & 368.304 & 1056.849 & 406.249 \\ 
                 \hhline{|~|~|~|~|~|~|~|~|~|~|-|~|-|~|~|~|~|~|~|~|~|~|~|}
					&        & 0     & 3.893e-02 & 2.699e-02 & 1.092e-01 & 9.197e-02 & 3.340e-03 & 1.398e-03 & 7.797e-04 & \vline \cellcolor[gray]{1.0} 1.119e-03 \vline & 1.703e-02 & \vline \cellcolor[gray]{1.0} \textbf{1.321e-02} \vline & \textbf{2.102e-02} & \textbf{1.875e-02} & 19 & 100 & 305 & 61 & 253.189 & 378.503 & 1102.067 & 412.168 \\
					\hline
					200 & 10000  & 0.1   & 2.442e-02 & 1.706e-02 & 7.392e-02 & 5.868e-02 & 7.643e-03 & 2.432e-03 & 1.449e-03 & 2.013e-03 & 2.885e-02 & 2.586e-02 & 2.552e-02 & 2.506e-02 & 18 & 110 & 214 & 60 & 252.154 & 393.617 & 826.667 & 411.387 \\
					&        & 0.01  & 2.402e-02 & 1.661e-02 & 6.925e-02 & 5.209e-02 & 2.346e-03 & 1.212e-03 & 1.586e-03 & 9.581e-04 & 1.405e-02 & 1.159e-02 & 1.817e-02 & 1.509e-02 & 20 & 106 & 62 & 60 & 254.167 & 389.364 & 421.967 & 412.413 \\
					&        & 0.001 & 2.377e-02 & 1.669e-02 & 6.744e-02 & 5.215e-02 & 2.202e-03 & 1.165e-03 & 1.644e-03 & 9.331e-04 & 1.378e-02 & 1.145e-02 & 1.792e-02 & 1.492e-02 & 21 & 117 & 62 & 60 & 256.723 & 401.675 & 422.083 & 411.567 \\
					&        & 0     & 2.432e-02 & 1.695e-02 & 6.787e-02 & 5.222e-02 & 2.446e-03 & 1.213e-03 & 1.835e-03 & 9.034e-04 & 1.389e-02 & 1.147e-02 & 1.802e-02 & 1.505e-02 & 22 & 99 & 62 & 61 & 256.975 & 378.268 & 421.892 & 416.442 \\
					\hline
					200 & 50000  & 0.1   & 1.007e-02 & \cellcolor[gray]{.8} 7.268e-03 & 2.080e-02 & 1.968e-02 & 8.088e-03 & 1.794e-03 & 6.689e-04 & 7.355e-04 & 2.662e-02 & 2.547e-02 & 2.155e-02 & 2.141e-02 & 19 & 130 & 208 & 197 & 254.379 & 434.248 & 840.326 & 800.123 \\
					&        & 0.01  & 9.390e-03 & 6.664e-03 & 2.213e-02 & 1.485e-02 & 1.352e-03 & 1.055e-03 & 9.224e-04 & 7.493e-04 & 1.071e-02 & 9.705e-03 & 1.222e-02 & 1.058e-02 & 26 & 155 & 153 & 60 & 263.740 & 492.920 & 750.784 & 423.843 \\
					&        & 0.001 & 9.451e-03 & 6.711e-03 & 1.664e-02 & 1.416e-02 & 1.302e-03 & 1.079e-03 & 1.514e-03 & 4.517e-04 & 1.046e-02 & 9.488e-03 & 1.204e-02 & 1.009e-02 & 28 & 142 & 62 & 150 & 267.072 & 471.474 & 450.813 & 727.345 \\
					&        & 0     & 9.415e-03 & \cellcolor[gray]{.8} 6.645e-03 & 1.628e-02 & 1.396e-02 & 1.315e-03 & 1.065e-03 & 1.461e-03 & 1.922e-04 & 1.045e-02 & 9.486e-03 & 1.195e-02 & 1.007e-02 & 22 & 142 & 62 & 177 & 258.197 & 447.639 & 451.086 & 834.282 \\
					\hline
					200 & 100000 & 0.1   & 7.182e-03 & 5.326e-03 & 1.227e-02 & 1.153e-02 & \textbf{1.013e-02} & 9.817e-04 & 6.327e-04 & \cellcolor[gray]{.8} \textbf{5.901e-04} & 2.637e-02 & \cellcolor[gray]{.8} 2.489e-02 & 2.084e-02 & 2.067e-02 & 18 & 124 & 348 & 381 & 254.638 & 435.492 & 1296.349 & 1379.963 \\
                  \hhline{|~|~|~|~|-|~|~|~|~|~|~|~|~|~|~|~|~|~|~|~|~|~|~|}
					&        & 0.01  & 6.984e-03 & \vline \cellcolor[gray]{1.0} 5.091e-03 \vline & 1.594e-02 & 9.382e-03 & 1.393e-03 & 1.089e-03 & 9.057e-04 & 6.660e-04 & 1.013e-02 & 9.352e-03 & 1.133e-02 & 9.655e-03 & 49 & 207 & 153 & 59 & 300.392 & 558.922 & 839.098 & 437.614 \\
                  \hhline{|~|~|~|~|-|~|~|~|~|~|~|~|~|~|~|~|~|~|~|~|~|~|~|}
					&        & 0.001 & 6.943e-03 & 5.030e-03 & 1.013e-02 & 8.310e-03 & 1.382e-03 & 1.101e-03 & 1.503e-03 & 3.327e-04 & 9.893e-03 & 9.096e-03 & 1.104e-02 & 9.155e-03 & 38 & 234 & 62 & 154 & 283.820 & 598.549 & 487.794 & 824.306 \\
                 \hhline{|~|~|~|~|~|~|~|~|~|~|-|~|-|~|~|~|~|~|~|~|~|~|~|}
					&        & 0     & 6.963e-03 & 5.004e-03 & 9.859e-03 & 8.093e-03 & 1.374e-03 & \textbf{1.084e-03} & \textbf{1.467e-03} &  \vline \cellcolor[gray]{.8} \textbf{\textit{1.343e-04}} \vline & 9.906e-03 & \vline \cellcolor[gray]{.8} \textit{9.078e-03} \vline & 1.099e-02 & 9.108e-03 & 37 & 178 & 62 & 181 & 282.477 & 512.666 & 487.917 & 967.058 \\
					\hline
			\end{tabular}}
	\end{minipage}}
	\caption{Non-linear law tested on $\Omega = (0,1)^2$ with different mesh sizes and data sets.} \label{tab:NL}
\end{table}
The following observations can be made: 

With respect to accuracy, all algorithms deliver comparable results with \eqref{eq:PS} 
being slightly superior regarding the objective and $\err_{H^1_0}$ and \eqref{eq:DR2} being slightly superior regarding $\err_{L^2}$.
To see this more clearly, we highlight the largest differences (compared to the result of \eqref{eq:PS} and \eqref{eq:DR2}, respectively)
concerning objective value and the errors $\err_{L^2}$ and $\err_{H^1_0}$ in boldface.

With regard to the effort, the simple projection method \eqref{eq:P} is nearly always the best. This concerns the 
number of iterations as well as the consumed computing time. 
Moreover, the differences are substantial. For instance, in the line highlighted in boldface, \eqref{eq:P} is up to approximately 5, 18, and 8 times 
faster than \eqref{eq:PS}, \eqref{eq:DR1}, and \eqref{eq:DR2}, respectively.

In accordance with our theoretical findings, the objective value decreases if the noise level $\bar s$ is reduced and/or 
if the data sample set is getting larger. Concerning the noise level, this reduction is however rather moderate. 
Even the largest reduction indicated with gray background is by less than 10~\%, though $\bar s$ is reduced from 0.1 to zero.
Regarding the errors $\err_{L^2}$ and $\err_{H^1_0}$, the dependency on the noise level is indifferent. 
The largest reduction, again indicated by gray background, is about 80~\% w.r.t.\ the $L^2$-error and 65~\% for the $H^1_0$-error. 
On the other hand, there are also instances, two of them marked in light gray, where the errors not only stagnate, 
but even increase with decreasing noise. In summary, the influence of the noise level appears to be limited and all algorithms 
behave comparatively robust w.r.t.\ noisy data.

The situation changes when the sample size is changed. 
With respect to the sample size, the reduction of the objective is more significant. Here, the largest reduction, when increasing the 
sample size from 5.000 to 10.000, marked by a box, is about 80~\%. In case of the $L^2$-error, the largest reduction 
is about one order of magnitude, while the largest reduction of $\err_{H^1_0}$ is approximately 35 \%, both again marked by boxes.
In summary, the sample size has a significantly larger impact on the accuracy of the results than the noise level.

The influence of the mesh size to the $L^2$- and the $H^1_0$-error is similar to that of ``classical'' finite element simulations.
We highlight the smallest $L^2$- and $H^1_0$-error for zero noise level and maximum $|\DD_k^\loc|$ in italic type.
In case of $\err_{L^2}$, the best result is obtained by \eqref{eq:DR2}, in case of $\err_{H^1_0}$ by \eqref{eq:PS}.
In Table~\ref{tab:error_Newton}, these values are compared with a ``classical'' finite element computation.
For the latter, we discretized \eqref{SDP} with $\kappa$ as defined in \eqref{eq:kappa} 
by choosing $U_h$ from \eqref{eq:Uh} as trial and test space. The resulting nonlinear system of equation is solved by Newton's method. The relative errors are denoted by $\err_{L^2}^{\textup{FE}}$ and $\err_{H^1_0}^{\textup{FE}}$, respectively.
We moreover present the experimental order of convergence defined by
\begin{equation*}
    \texttt{EOC}_{L^2}^{\textup{FE}}
    := \frac{\log(\err_{L^2}^{\textup{FE}}(h_1)) - \log(\err_{L^2}^{\textup{FE}}(h_2))}{\log(h_1) - \log(h_2)},
\end{equation*}
where $h_1$ and $h_2$ are two mesh sizes and $\err_{L^2}^{\textup{FE}}(h_i)$, $i=1,2$, the associated relative $L^2$-errors. 
Furthermore, $\texttt{EOC}_{H^1_0}^{\textup{FE}}$ is defined analogously and $\texttt{EOC}_{L^2}^{\textup{DR2}}$ 
and $\texttt{EOC}_{H^1_0}^{\textup{PS}}$ denote the respective orders of convergence generated by the data driven approach with 
\eqref{eq:DR2} and \eqref{eq:PS}, respectively.
\begin{table}[!h]
	\resizebox{\textwidth}{!}{%
	\begin{tabular}{|c|c|c|c|c|c|c|c|c|}
		\hline
		N & $\err_{L^2}^{\textup{FE}}$ & $\texttt{EOC}_{L^2}^{\textup{FE}}$ 
		& $\err_{L^2}^{\textup{DR2}}$ & $\texttt{EOC}_{L^2}^{\textup{DR2}}$ 
		& $\err_{H^1_0}^{\textup{FE}}$ & $\texttt{EOC}_{H^1_0}^{\textup{FE}}$ 
		& $\err_{H^1_0}^{\textup{PS}}$ & $\texttt{EOC}_{H^1_0}^{\textup{PS}}$\\
		\hline
		50 & 1.601e-03 & --- & 1.598e-03 & --- & 3.143e-02 & ---  & 3.174e-02 & --- \\
		100 & 4.004e-04  & 1.9995 & 3.616e-04 & 2.1438 & 1.571e-02 & 1.0005 & 1.637e-02 & 0.9552\\
		200 & 1.001e-04 & 2.0000 & 1.343e-04 & 1.4289 & 7.854e-03  & 1.0002 & 9.078e-03 & 0.8506\\
		\hline
	\end{tabular}}
	\caption{Relative errors and experimental orders of convergence for the ``classical'' finite element solution and the 
	best results of the data-driven approach 
	in dependence of the mesh sizes $h =\frac{\sqrt{2}}{N}$.} \label{tab:error_Newton}
\end{table}
As the theory predicts, cf.\ e.g.\ \cite{BrennerScott1994}, 
Table~\ref{tab:error_Newton} shows a quadratic and a linear order of convergence for the relative errors 
$\err_{L^2}^{\textup{FE}}$ and $\err_{H^1_0}^{\textup{FE}}$, respectively.
We moreover observe that the errors produced by the data driven approach are more or less of the same size as the 
``classical'' finite element error except for the smallest mesh with $N=200$, where the error caused by the sample size 
becomes predominant compared to the error induced by the mesh.

To summarize, it is to be noted that the accuracy of all algorithms improve, if the noise level is reduced, the data sample set 
is enlarged, and the mesh is refined, with the noise level having the smallest impact on the accuracy.
We moreover observe that all algorithms yield satisfactorily results w.r.t.\ the accuracy, the best results being even comparable to 
a ``classical'' finite element computation.
Concerning the performance of the algorithms, the original projection algorithm \eqref{eq:P} turns out to be superior in the sense that it
provides an accuracy similar to the other algorithms with the fastest computing time. 
The proximal gradient method defined by \eqref{eq:PS} in average returns the most accurate results 
(in particular w.r.t.\ the $H^1_0$-seminorm) but with a greater computational effort. 
Both variants of the Douglas-Rachford algorithm also provide comparable results but one has to be cautious 
concerning the termination criterion due to the lack of fixed points.


\section{Conclusion and Outlook}

In this paper, we studied a finite element discretization of the data driven approach as introduced \cite{KirchdoerferOrtiz2016}, 
where we focused on a stationary scalar diffusion problem. We showed that the finite element error analysis can be 
incorporated into the data convergence analysis of \cite{CMO2020} as long as finite elements are used where a vanishing discretized divergence 
implies that the continuous divergence vanishes, too, cf.\ Lemma~\ref{lem:saddleh}\ref{it:divdivh}. 
In the conductivity example, i.e., our stationary scalar diffusion problem, the construction of such elements is comparatively 
simple and, for instance, Raviart-Thomas elements do the job, see Proposition~\ref{prop:RT}. 
The situation changes, if one turns to problems in elasticity, where the symmetry of the stress tensor significantly complicates the construction 
of such elements, see for instance \cite[Section~4]{braess} and \cite{arnoldbrezzidouglas}.
The considerations in Section~\ref{sec:whyHdiv} however indicate that the use of ``classical'' piecewise linear and continuous 
finite elements might not be feasible in context of the data driven approach, an aspect that gives rise to future research. 
Nevertheless, if $H(\div)$-conforming finite elements fulfilling Assumptions~\ref{assu:discret} are used, then, 
under suitable assumptions on the approximation of the local data set, see \eqref{eq:DDlocapprox}, we obtain the same 
approximation results as in \cite{CMO2020} so that (subsequences of) minimizers of the finite dimensional problems \eqref{eq:Pk2} converge in data 
to elements of the intersection $\overline{\DD} \cap \EE$.

Another issue concerns the computation of minimizers of \eqref{eq:Pk2}. As seen in Section~\ref{sec:qsap}, the discretized 
data driven problem \eqref{eq:Pk2} is equivalent to a quadratic semi-assignment problem and, as such, NP-hard, see Remark~\ref{rem:qsap}. 
There is thus no hope to find an efficient algorithm for the computation of minimizers. 
We therefore presented two heuristic approaches, one based on projections on $\EE_h$ and $\DD_h$, the other one based on 
a local search in combination with model order reduction. It turns out that, the local search is not competitive, at least for the example of Fourier's law. 
There are however plenty of parameters to adjust, such as the size of the POD basis and the neighborhood within the local search, 
and it requires further investigations to analyze if a smart choice of the parameters could result in a competitive algorithm.
With regard to the projection-type methods, the most simple alternating projection algorithm according to \cite{KirchdoerferOrtiz2016} 
turns out to be advantageous with respect to the ratio of accuracy and computational time. With regard to accuracy only, 
the modified projection method \eqref{eq:PS} and a variant of the Douglas-Rachford algorithm delivered the best results. 
If the noise level is zero and the sample size is large enough, then these results are comparable to a ``classical'' finite element 
simulation with known material law. There is however no theoretical evidence for a convergence of the projection-based methods, 
even not to any kind of fixed-point, not to mention a minimizer of \eqref{eq:Pk2}, and the examples in \cite{Kanno2019} show that this is indeed an issue.
The robustification of projection-based methods for a reliable solution of \eqref{eq:Pk2} is probably one of the most important open problems 
in the context of data driven numerics.

\bibliographystyle{plain}

\end{document}